\documentclass{amsart}

\newtheorem{lemma}{Lemma}[section]
\newtheorem{theorem}[lemma]{Theorem}

\newtheorem{proposition}[lemma]{Proposition}

\newcommand{\intbar}{\mathop{\int\makebox(-13.5,0){\rule[4pt]{.7em}{0.3pt}}%
\kern-6pt}\nolimits}  

\newcommand{\smallintbar}{\mathop{\int\makebox(-10,0)%
{\rule[4pt]{.6em}{0.3pt}}\kern-6pt}\nolimits}  

\makeatletter
\def\proof{\@ifnextchar[\opargproof{\opargproof[\bf Proof. ]}}
\def\opargproof[#1]{\par\noindent {\bf #1}}

\makeatother

\begin{document}

\title[Heat flow with critical exponential nonlinearity]
{The heat flow with a critical exponential nonlinearity}
\author{Tobias Lamm}
\author{Fr\'ed\'eric Robert}
\author{Michael Struwe}
\address[Tobias Lamm]{Department of Mathematics, University 
of British Columbia, Vancouver, BC V6T 1Z2, Canada}
\email{tlamm@math.ubc.ca}
\address[Fr\'ed\'eric Robert]{Universit\'e de Nice-Sophia Antipolis, 
Laboratoire J.-A.Dieudonn\'e, Parc Valrose, 06108 Nice Cedex 2, France}
\email{frobert@math.unice.fr}
\address[Michael Struwe]{Departement Mathematik\\ETH-Z\"urich\\CH-8092 Z\"urich}
\email{michael.struwe@math.ethz.ch}

\date{May 22nd, 2009}

\begin{abstract} We analyze the possible concentration behavior of heat flows
related to the Moser-Trudinger energy and derive quantization results completely 
analogous to the quantization results for solutions of the corresponding elliptic 
equation.
As an application of our results we obtain the existence of critical points of the
Moser-Trudinger energy in a supercritical regime.
\end{abstract} 

\maketitle

\section{Introduction} 
On any bounded domain $\Omega \subset {\mathbb R}^2$
the Moser-Trudinger energy functional 
\begin{equation*}
     E(u)=\frac12\int_{\Omega}(e^{u^2}-1)\, dx
\end{equation*}
for any $\alpha \le 4 \pi$ admits a maximizer in the space 
\begin{equation} \label{1.00}
      M_{\alpha} = \{u \in H^1_0(\Omega);\, u \ge 0, ||\nabla u||^2_{L^2}=\alpha \},
\end{equation}
corresponding to a solution $0 < u \in H^1_0(\Omega)$ of the equation 
\begin{equation} \label{1.01}
   -\Delta u = \lambda u e^{u^2} \hbox{ in } \Omega
\end{equation}
for some $\lambda > 0$; see \cite{Carleson-Chang} and \cite{Flucher}.
Moreover, when $\Omega$ is a ball numerical evidence \cite{Monahan} shows that 
for small $\alpha>4\pi$ there exists a pair of critical points of $E$ in 
$M_{\alpha}$,
corresponding to a relative maximizer and a saddle point of $E$, respectively.
However, standard variational techniques fail in this ``supercritical'' energy range
and ad hoc methods devised to remedy the situation so far have only been partially 
succesful in producing the expected existence results; compare 
\cite{Struwe84}, \cite{Struwe2000}. 
As in various other geometric variational problems a
flow method might turn out to be more useful in this regard.

Given a smooth function $0 \le u_0 \in H^1_0(\Omega)$,
we consider smooth solutions $u=u(t,x)$ to the equation
\begin{equation} \label{1.1}
   u_t e^{u^2} = \Delta u + \lambda u e^{u^2} \hbox{ in } [0,\infty[ \times \Omega
\end{equation}
with initial and boundary data 
\begin{equation} \label{1.2}
   u(0) = u_0, \quad u = 0 \hbox{ on } [0,\infty[ \times \partial \Omega.
\end{equation}
The function $\lambda= \lambda(t)$ may be determined so that the  
Dirichlet integral of $u$ is preserved along the flow. As we shall see, also 
the case where the volume of the evolving metric $g=e^{u^2}g_{{\mathbb R}^2}$ 
is fixed gives rise to interesting applications, and both constraints can
easily be analyzed in parallel.

\subsection{Fixed volume}
Fixing the volume is equivalent to the constraint 
\begin{equation} \label{1.5}
   E(u(t)) =E(u_0) =: c_0
   \hbox{ for all } t,
\end{equation}
which can be achieved by imposing the condition
\begin{equation} \label{1.3}
   \frac{d}{dt}E(u(t)) = \int_{\Omega} u u_t e^{u^2} dx = 
   \lambda \int_{\Omega} u^2 e^{u^2} dx - \int_{\Omega}|\nabla u|^2 \, dx = 0.  
\end{equation}
Clearly, we may assume that $u_0$ does not vanish identically and that $c_0 > 0$; 
otherwise $u \equiv 0$ is the unique smooth solution to \eqref{1.1} - \eqref{1.5}
for any choice of $\lambda(t)$.

Note that when we multiply \eqref{1.1} with $u_t$ and use \eqref{1.3}, upon 
integrating by parts we obtain the relation 
\begin{equation} \label{1.6}
    \int_{\Omega} u_t^2 e^{u^2} dx 
   + \frac12 \frac{d}{dt}\left(\int_{\Omega}  |\nabla u|^2  \, dx\right) 
   = \lambda\frac{d}{dt}E(u(t)) = 0;
\end{equation}
that is, the flow \eqref{1.1} - \eqref{1.5} 
may be regarded as the gradient flow (with respect to the metric $g$)
for the Dirichlet energy with the critical exponential constraint \eqref{1.5}.

Equation \eqref{1.3} and the energy inequality \eqref{1.6}
imply the uniform bound
\begin{equation} \label{2.2}
   \lambda \int_{\Omega} u^2 e^{u^2} dx 
   = \int_{\Omega}  |\nabla u|^2  \, dx \le \int_{\Omega}  |\nabla u_0|^2  \, dx
   =: \Lambda_0.
\end{equation}
Since we can easily estimate 
$e^a \le 1 + 4a$ for $0 \le a \le 1/4$, we have
\begin{equation} \label{2.3}
  \begin{split}
    & \int_{\Omega} u^2 e^{u^2} dx 
      = \int_{\Omega} u^2 (e^{u^2} - 1) \, dx + \int_{\Omega} u^2 \, dx\\ 
    & \ge \frac14 \int_{\Omega} (e^{u^2} - 1) \, dx 
       - \frac14 \int_{\{x \in \Omega; u \le 1/2\}} (e^{u^2} - 1) \, dx
       + \int_{\Omega} u^2 \, dx \ge \frac{E(u)}{2} \ge \frac{c_0}{2},
  \end{split}  
\end{equation}
for all $t$. Therefore, recalling that $c_0 > 0$, from \eqref{2.2} we deduce
that with the constant $\lambda_0 =2\Lambda_0/c_0 > 0$ there holds 
\begin{equation} \label{2.4}
   0 < \lambda(t) \le \lambda_0 \hbox{ for all } t \ge 0.
\end{equation}
Finally, the maximum principle yields that $u \ge 0$.

\subsection{Constant Dirichlet integral}
If, on the other hand, we choose $\lambda$ so that
\begin{equation}\label{0.0}
   \frac12\frac{d}{dt}\bigg(\int_{\Omega}  |\nabla u|^2  \, dx\bigg)
   = - \int_{\Omega} u_t\Delta u \, dx 
   =  \lambda \int_{\Omega} |\nabla u|^2 \, dx
      - \int_{\Omega} |\Delta u|^2e^{-u^2}dx = 0,
\end{equation}
for a solution of \eqref{1.1}, \eqref{1.2} satisfying
\eqref{0.0} the Dirichlet integral is preserved; that is, 
\begin{equation} \label{0.1}
  \int_{\Omega} |\nabla u|^2  \, dx = \int_{\Omega}|\nabla u_0|^2 \, dx = \Lambda_0 ,
\end{equation}
In this case, from \eqref{1.6} we find the equation
\begin{equation} \label{0.2}
    \int_{\Omega} u_t^2 e^{u^2} dx = \lambda\frac{d}{dt}E(u(t)),
\end{equation}
and \eqref{1.1}, \eqref{1.2} with the constraint \eqref{0.1} turns into the 
(positive) gradient flow for the Moser-Trudinger
energy with prescribed Dirichlet integral. Again clearly we may assume that 
$\Lambda_0>0$.

Recalling the identity 
\begin{equation*}
   \frac{d}{dt}E(u(t)) = \lambda \int_{\Omega} u^2 e^{u^2} dx -
   \int_{\Omega}|\nabla u|^2 \, dx  
\end{equation*}
\eqref{0.2} and \eqref{2.3}, for any $t$ we have 
\begin{equation} \label{1.4}
  \frac{c_0}{2} \int_0^t \lambda \, dt \le \Lambda_0 t + E(u(t)) - E(u_0),
\end{equation}
where $c_0=E(u_0)\le E(u(t))$ for all $t \ge 0$. Similarly, from \eqref{0.2} we obtain 
\begin{equation} \label{0.3}
    \int_0^t\big(\lambda^{-1}\int_{\Omega} u_t^2 e^{u^2}\, dx\big)\, dt =
    E(u(t)) - E(u_0).
\end{equation}
Hence we can hope to obtain bounds for solutions of \eqref{1.1}, \eqref{1.2},
\eqref{0.1} whenever the Moser-Trudinger energy is bounded along the flow.

\subsection{Results}
Building on previous results from \cite{Adimurthi-Struwe}, \cite{Druet}, and
\cite{Struwe07}, in this paper we establish the 
following result for the flow \eqref{1.1}, \eqref{1.2} with either the 
constraint \eqref{1.5} or the constraint \eqref{0.1}.

\begin{theorem} \label{thm1.1} 
For any $c_0 > 0$ and any smooth initial data $0 \le u_0 \in H^1_0(\Omega)$ 
satisfying \eqref{1.5} the evolution problem \eqref{1.1} - \eqref{1.5} 
admits a unique smooth solution $u \ge 0$ for all $t > 0$. 
Likewise, for any smooth $0 \le u_0 \in H^1_0(\Omega)$ satisfying \eqref{0.1}
for a given $\Lambda_0>0$ the evolution problem \eqref{1.1}, \eqref{1.2},
\eqref{0.1} admits a unique smooth solution $u \ge 0$ for small $t > 0$ 
which can be continued smoothly for all $t > 0$, provided that $E(u(t))$ remains 
bounded. 
In both cases, for a suitable sequence $t_k \rightarrow \infty$ the functions 
$u(t_k) \rightarrow u_{\infty}$ weakly in $H^1_0(\Omega)$, where 
$u_{\infty}\in H^1_0(\Omega)$ is a solution to the problem \eqref{1.01}
for some constant $\lambda_{\infty} \ge 0$. Moreover, either  
$u(t_k) \rightarrow u_{\infty}$ strongly in $H^1_0(\Omega)$,
$\lambda_{\infty} > 0$, and $0 < u_{\infty}\in H^1_0(\Omega)$ 
satisfies, respectively, \eqref{1.5} or \eqref{0.1}, or there exist 
$i_* \in {\mathbb N}$ and 
points $x^{(i)} \in \overline{\Omega}$, $l_i \in {\mathbb N}$, $1 \le i \le i_*$, 
such that as $k \rightarrow \infty$
we have 
\begin{equation*}
  |\nabla u(t_k)|^2dx \overset{w^*}{\rightharpoondown} |\nabla u_{\infty}|^2dx 
  + \sum_{i=1}^{i_*} 4\pi l_i\delta_{x^{(i)}}
\end{equation*}
weakly in the sense of measures. By \eqref{2.2} or \eqref{0.1} then
necessarily $4\pi \sum_{i=1}^{i_*}l_i \le \Lambda_0$.
\end{theorem}

The quantization result in the case of divergence of the flow 
relies on the precise microscopic description of blow-up given in
Sections 4 and 5; see in particular Theorems \ref{thm4.1} and \ref{thm5.1}.
Their derivation will take up the major part of this paper.
These results are in complete analogy with the 
results of Adimurthi-Struwe \cite{Adimurthi-Struwe} and Druet \cite{Druet}
for solutions of the corresponding elliptic equation \eqref{1.01}.

Note that our equation \eqref{1.1} is similar to the equation for 
scalar curvature flow. In $m=2$ space dimensions this latter flow corresponds 
to the Ricci flow studied by Hamilton \cite{Hamilton} 
and Chow \cite{Chow}; see \cite{Struwe02} for a more analytic approach.
For $m\ge 3$ the scalar curvature flow is the Yamabe flow analyzed by 
Ye \cite{Ye}, Schwetlick-Struwe \cite{Schwetlick-Struwe}, and 
Brendle \cite{Brendle05}, \cite{Brendle07}.
Surprisingly, these geometric flows can be shown to always converge.
This stands in contrast to the behavior of semi-linear parabolic flows 
with polynomial nonlinearities that were studied for instance by Giga \cite{Giga} or, 
more recently, Tan \cite{Tan}, where the term involving the time derivative
is not modulated by the solution and where we may observe blow-up in finite time. 

Even though our equation \eqref{1.1} does not seem to have an obvious geometric
interpretation, we are able to show that its blow-up behavior (as long as the 
energy stays bounded) is rigidly determined by the properties of Liouville's equation
in the plane, that is, by the properties of Gauss' equation on $S^2$.
We do not know if the analogy with the $2$-dimensional Ricci flow extends even further; 
in particular, we do not know if all solutions to either 
\eqref{1.1} - \eqref{1.5} or \eqref{1.1}, \eqref{1.2} with the constraint \eqref{0.1}
and having uniformly bounded energy smoothly converge as $t \rightarrow \infty$. 

Even so Theorem \ref{thm1.1} is sufficient to yield existence of saddle-point 
solutions for \eqref{1.01} in supercritical regimes of large energy. 
In the final Section 6 we illustrate this with two examples where we use 
\eqref{1.1}, \eqref{1.2} with either the constraint \eqref{1.5} or \eqref{0.1}.
For a domain $\Omega \subset {\mathbb R}^2$ with $vol(\Omega) = \pi$ we define 
$$
  c_{4\pi}(\Omega):= 
  \sup_{u \in H^1_0 (\Omega); ||\nabla u||^2_{L^2(\Omega)} \le 4\pi} E(u).
$$
Note that we always have $c_{4\pi}(\Omega)\le c_{4\pi}(B_1(0)=:c_*$.
Our first result then provides the following analogue of  Coron's result \cite{Coron}; 
it also is related to Theorem 1.1 in \cite{Struwe2000}.

\begin{theorem} \label{thm1.2}
For any $c^* > c_*$ there are numbers $R_1 > R_2 > 0$ with the following
property. Given any domain
$\Omega \subset {\mathbb R}^2$ with $vol(\Omega) = \pi$ containing the annulus 
$B_{R_1} \setminus B_{R_2} (0)$ and such that $0 \notin \overline{\Omega}$,
for any constant $c_0$ with $c_{4\pi}(\Omega)< c_0< c^*$ problem \eqref{1.01} 
admits a positive solution $u$ with $E(u) =c_0$.
\end{theorem}

Our second result completes Theorem 1.8 from \cite{Struwe84}.

\begin{theorem}\label{thm1.3}
There exists a number $\alpha_1 \in ]4\pi, 8\pi]$ such that for any 
$4\pi<\alpha<\alpha_1$ there exists a pair of solutions 
$\underline{u}, \overline{u} \in M_{\alpha}$ of \eqref{1.01} with 
$0 < E(\underline{u}) < E(\overline{u})$.
\end{theorem} 

In \cite{Struwe84} the existence of a pair 
of solutions of \eqref{1.01} only was shown for almost every $4\pi< \alpha <\alpha_1$.

\section{Global existence}
Let $u(t)$ be a solution of \eqref{1.1}, \eqref{1.2} with either
the constraint \eqref{1.5} or \eqref{0.1}. In the latter case we also assume
that $E(u(t))$ remains bounded. 
For any $t \ge 0$ let $m(t) = ||u(t)||_{L^{\infty}}$. 
Writing equation \eqref{1.1} in the form
\begin{equation*}
   u_t - e^{-u^2} \Delta u = \lambda u
   \hbox{ in } [0,\infty[ \times \Omega
\end{equation*}
and observing that $\Delta u \le 0$ at any point where $u(t)$ 
achieves its maximum, we conclude that the supremum of the function 
$\tilde{u}(t) = e^{-\int_0^t\lambda(s)ds}u(t)$
is non-increasing in time. That is, for any $0 \le t_0 \le t < \infty$ we have
\begin{equation} \label{2.5}
   m(t) \le e^{\int_{t_0}^t\lambda(s)ds}m(t_0).
\end{equation}
Together with \eqref{2.4}, \eqref{1.4} this immediately gives the following 
result.

\begin{lemma}\label{lemma0.1}
Suppose that $E(u(t))$ is uniformly bounded. Then there exist 
constants $\lambda_1 > 0$, $C_1$ depending on $u_0$ such 
that for any $t \ge 0$ we have 
\begin{equation*}
   ||u(t)||_{L^{\infty}} \le e^{\int_0^t\lambda(s)ds} ||u_0||_{L^{\infty}}
   \le C_1e^{\lambda_1 t} ||u_0||_{L^{\infty}}.
\end{equation*}
\end{lemma}

Existence of a unique smooth solution on any finite time interval now follows 
from standard results on uniformly parabolic equations.  

\section{Asymptotic behavior}
\subsection{Weak subconvergence}
First consider the constraint \eqref{1.5}.
Integrating in time, from \eqref{1.6} we then obtain
\begin{equation} \label{3.1}
   \int_0^{\infty} \int_{\Omega} u_t^2 e^{u^2} dx \, dt 
   \le \frac12 \int_{\Omega}  |\nabla u_0|^2 \, dx.
\end{equation}
Hence we can find a sequence $t_k \rightarrow \infty$ such that
\begin{equation} \label{3.1a}
    \int_{\{t_k\} \times \Omega} u_t^2 e^{u^2} dx \rightarrow 0 
    \hbox{ as } k \rightarrow \infty.
\end{equation}

In view of \eqref{2.4} and \eqref{2.2} from any such 
sequence $(t_k)$ we may extract a subsequence such that
$\lambda_{\infty} = \lim_{k \rightarrow \infty}\lambda(t_k)$ exists and
such that, in addition, $u_k = u(t_k) \rightharpoondown u_{\infty}$ 
weakly in $H^1_0(\Omega)$ and pointwise almost everywhere as 
$k \rightarrow \infty$. From \eqref{2.2} by means of
the Vitali convergence theorem 
we then deduce that for a further subsequence
the terms $\lambda u e^{u^2}$, evaluated
at $t = t_k$, converge to $\lambda_{\infty} u_{\infty} e^{u^2_{\infty}}$
in $L^1(\Omega)$. Thus,
upon passing to the limit $k \rightarrow \infty$ in \eqref{1.1} we see that
$u_{\infty}$ is a (weak) solution to equation \eqref{1.01}. 
But since $u_{\infty} \in H^1_0(\Omega)$, from the Moser-Trudinger 
inequality it follows that $u_{\infty} e^{u^2_{\infty}} \in L^p(\Omega)$
for any $p < \infty$, and $u_{\infty}$ is, in fact, smooth. 

Similarly, in the case of the the constraint \eqref{0.1}, assuming that $E(u(t))$ is 
uniformly bounded from above along the flow \eqref{1.1}, \eqref{1.2},
from \eqref{0.3} we obtain the bound 
\begin{equation} \label{0.3a}
    \int_0^{\infty}\bigg(\lambda^{-1}\int_{\Omega} u_t^2 e^{u^2}\, dx\bigg)\, dt \le 
    \lim_{t \rightarrow \infty}E(u(t)) - E(u_0) < \infty,
\end{equation}
and we can find a sequence $t_k \rightarrow \infty$ such that
\begin{equation} \label{0.4}
   \lambda(t_k)^{-1} \int_{\{t_k\} \times \Omega} u_t^2 e^{u^2} dx \rightarrow 0 
   \hbox{ as } k \rightarrow \infty.
\end{equation} 

Necessarily the sequence $(\lambda(t_k))$
is bounded. 
Indeed, upon multiplying \eqref{1.1} by $u$
we infer that at time $t_k$ with error $o(1) \rightarrow 0$ we have
\begin{equation*}
   \begin{split}
     \lambda\int_{\Omega} & u^2 e^{u^2}dx = \int_{\Omega} |\nabla u|^2 dx 
     +\int_{\Omega} u u_t e^{u^2} dx
    \end{split}
\end{equation*}
But by \eqref{0.4} and H\"older's inequality, at time $t = t_k$ with 
error $o(1) \rightarrow 0$ as $k \rightarrow \infty$ we can estimate
\begin{equation}\label{0.7}
    \big|\int_{\Omega} u u_t e^{u^2} dx\big|^2 \le \lambda \int_{\Omega} u^2 e^{u^2}dx
    \cdot \lambda^{-1}\int_{\Omega} u_t^2 e^{u^2}\, dx 
    = o(1) \lambda \int_{\Omega} u^2 e^{u^2}dx
\end{equation}
and we have
\begin{equation} \label{0.5}
   (1+o(1))\lambda \int_{\Omega} u^2 e^{u^2} dx = \int_{\Omega}|\nabla u|^2 \, dx 
   = \Lambda_0. 
\end{equation}
Our claim now follows from \eqref{2.3}.
Note that, in particular, the approximate identity \eqref{2.2} 
thus also holds in the case of the constraint \eqref{0.1}. 

\subsection{The case when $u$ is bounded}
If in addition we assume that the function $u$ is uniformly bounded we 
find that any sequence $(u_k)$ as above is bounded in $H^2(\Omega)$
and hence possesses a subsequence such that 
$u_k \rightarrow u_{\infty}$ strongly in $H^1_0(\Omega)$ 
as $k \rightarrow \infty$. Hence $u_{\infty}\in H^1_0(\Omega)$ 
satisfies, respectively, \eqref{1.5} or \eqref{0.1}, 
and $u_{\infty} > 0$ by the maximum principle.

In the case of the constraint \eqref{1.5}, and provided that $u$ is bounded,
we can even show relative compactness of the sequence $u_k = u(t_k)$
for {\it any} sequence $t_k \rightarrow \infty$ .

\begin{proposition}\label{prop3.1}
Let $u$ solve \eqref{1.1} - \eqref{1.5}. Suppose that 
there exists a uniform constant $M > 0$ such that 
$u(t,x) \le M$ for all $x \in \Omega$ and all $t\ge 0$.
Then any sequence $u_k = u(t_k)$ with $t_k \rightarrow \infty$
has a strongly convergent subsequence.
\end{proposition}

\begin{proof}
It suffices to show that under the assumptions of the Proposition the 
convergence in \eqref{3.1a} can be improved to be uniform in time. 
To show this we use \eqref{1.1} to calculate
\begin{equation*}
\begin{split}
 u_{tt}&=\lambda_t u+\lambda u_t-2uu_te^{-u^2}\Delta u+e^{-u^2}\Delta u_t\\
 & =\lambda_t u+\lambda u_t+e^{-u^2}\Delta u_t-2u u_t^2+2\lambda u^2 u_t.
\end{split}
\end{equation*}
Thus we obtain 
\begin{equation*}
\begin{split}
 \frac{1}{2} \frac{d}{dt}\bigg(\int_\Omega & u_t^2 e^{u^2}dx \bigg) 
 = \int_\Omega u_t u_{tt}e^{u^2}dx+\int_\Omega u_t^3 u e^{u^2}dx \\
 & = \lambda_t \int_\Omega u u_t e^{u^2}dx 
   + \lambda \int_\Omega u_t^2 e^{u^2}dx +\int_\Omega u_t \Delta u_t\, dx\\
 & \quad\quad - 2\int_\Omega u u_t^3e^{u^2}dx
   + 2\lambda \int_\Omega u^2u_t^2 e^{u^2}dx \\
\end{split}
\end{equation*}
By \eqref{1.3} the first term on the right vanishes. Moreover, we may use the 
fact $u_t=0$ on $\partial \Omega$ to integrate by parts in the third term.
Also using H\"olders inequality and Sobolev's embedding 
$W^{1,2}\hookrightarrow L^4$ then with constants $C=C(M)$ we find
\begin{equation} \label{1.15}
\begin{split}
 \int_\Omega & |\nabla u_t|^2 dx 
 + \frac{1}{2} \frac{d}{dt}\bigg(\int_\Omega u_t^2 e^{u^2}dx \bigg) \\
 & \le C \int_\Omega u_t^2 e^{u^2}dx
 + C \bigg(\int_\Omega u_t^2 e^{u^2}dx\bigg)^{\frac{1}{2}}
     \bigg(\int_\Omega u_t^4dx\bigg)^{\frac{1}{2}}\\
 & \le C \int_\Omega u_t^2 e^{u^2}dx
 + C_1 \bigg(\int_\Omega u_t^2 e^{u^2}dx\bigg)^{\frac{1}{2}}
         \int_\Omega \big(|\nabla u_t|^2+u_t^2e^{u^2}\big)dx.
\end{split}
\end{equation}
To proceed, we use an argument similar to \cite{Struwe02}, p. 271. 
Given any number $\varepsilon_0>0$, by \eqref{3.1} there exist 
arbitrary large times $t_0$ such that 
\begin{equation} \label{1.16}
\int_{\{t_0\}\times\Omega} u_t^2 e^{u^2}dx <\varepsilon_0.
\end{equation}
For any such $t_0$ we may choose a maximal $t_0\le t_1\le \infty$ such that 
\begin{equation} \label{1.17}
 \sup_{t_0\le t\le t_1} \int_{\{t\}\times\Omega} u_t^2 e^{u^2}dx
 \le 2\varepsilon_0.
\end{equation}
If we now fix $\varepsilon_0=\frac{1}{16C_1^2}$, from \eqref{1.15} at any
time $t\in [t_0,t_1]$ we obtain
\begin{equation} \label{1.18}
 \frac{1}{2}\frac{d}{dt}\bigg(\int_\Omega u_t^2 e^{u^2}dx\bigg)
 \le C\int_\Omega u_t^2 e^{u^2}dx.
\end{equation}
Integrating from $t_0$ to $t$ and using \eqref{3.1}, for any $t\in [t_0,t_1]$ 
we get
\begin{equation} \label{1.19}
\begin{split}
  \int_{\{t\}\times\Omega} u_t^2 e^{u^2}dx
  & \le \int_{\{t_0\}\times\Omega}u_t^2 e^{u^2}dx
  + C \int_{t_0}^\infty \int_\Omega u_t^2 e^{u^2}dx < 2\varepsilon_0,
\end{split}
\end{equation}
if $t_0$ is large enough. For such $t_0$ then $t_1=\infty$, and we conclude 
\begin{equation} \label{1.21}
\begin{split}
  \limsup_{t\rightarrow \infty} & \int_{\{t\}\times\Omega}u_t^2 e^{u^2}dx\\ 
  & \le \liminf_{t_0\rightarrow \infty}
        \bigg(\int_{\{t_0\}\times\Omega} u_t^2 e^{u^2}dx
    + C \int_{t_0}^\infty \int_\Omega u_t^2 e^{u^2}dx\bigg)=0.
\end{split}
\end{equation}
Using again the assumption that $u$ is uniformly bounded this directly implies that
\begin{equation}\label{1.22}
     \limsup_{t\rightarrow \infty}||u(t)||_{H^2} < \infty
\end{equation}
and hence the claim.
\end{proof}

\section{Blow-up analysis}
It remains to analyze the blow-up behavior of a solution $u$ to 
\eqref{1.1}, \eqref{1.2} satisfying either \eqref{1.5} or \eqref{0.1} 
in the case when $u$ is unbounded. As we shall see,
this can be done in complete analogy with the corresponding time-independent
problem. The key is the following lemma, which refines our above 
choice of $(t_k)$. 
 
\begin{lemma}\label{lemma4.1}
Suppose that $\limsup_{t\rightarrow \infty} ||u(t)||_{L^{\infty}} = \infty$
and that $E(u(t)) \le E_{\infty}$ for some constant $E_{\infty} < \infty$.
Then there is a sequence $t_{k} \rightarrow \infty$ with associated numbers 
$\lambda_k =\lambda(t_k)\rightarrow \lambda_{\infty}\ge 0$
such that $u(t_{k}) \rightharpoondown u_\infty$ weakly in $H^1_0(\Omega)$ 
as $k \rightarrow \infty$ and
\begin{equation*}
   ||u(t_k)||_{L^{\infty}} \rightarrow \infty,\ 
   \lambda_k^{-1} \int_{\{t_k\}\times \Omega}|u_t|^2 e^{u^2}dx\, dt 
   \rightarrow 0\; .
\end{equation*}
\end{lemma}

\begin{proof}
Suppose by contradiction that there exist $t_0 \ge 0$ and a 
constant $C_0 > 0$ such that for all $t \ge t_0$ either there holds
\begin{equation*}
     m(t)=||u(t)||_{L^{\infty}}\le C_0, 
\end{equation*}
or 
\begin{equation} \label{4.3a} 
    \lambda(t) \le C_0 \int_{\{t\}\times \Omega} |u_t|^2 e^{u^2}dx\; .   
\end{equation}

Consider first the constraint \eqref{1.5}. 
If $m(t) > C_0$ for all $t \ge t_0$, then \eqref{4.3a} holds for all such $t$
and upon integrating in time from \eqref{1.6} for any $t \ge t_0$ we obtain
\begin{equation}\label{4.3c}
  \begin{split}
    \int_{t_0}^{t} \lambda(s)ds 
    & \le C_0 \int_{0}^{\infty}\int_{\Omega} |u_t|^2 e^{u^2}dx\, dt 
    \le \frac{C_0\Lambda_0}{2} =:C_1 < \infty.
  \end{split}
\end{equation}
Applying \eqref{2.5} to the shifted flow $u(t-t_0)$ we find 
$\sup_{t \ge t_0} m(t) \le m(t_0) e^{C_1} < \infty$,
contrary to assumption. 

If for some $t_0 \le t_1 < t_2 \le \infty$ and all $t_1 < t < t_2$
we have $m(t_1)=C_0< m(t)$, then \eqref{4.3a} holds for all such $t$ and 
we obtain \eqref{4.3c} with $t_1$ replacing $t_0$ for all $t \in [t_1,t_2]$. 
Applying \eqref{2.5} to the shifted flow $u(t-t_1)$, for any such
$t_0 \le t_1 < t_2 \le \infty$ we obtain the bound
$\sup_{t_1 < t \le t_2} m(t) \le C_0 e^{C_1} < \infty$, 
again contradicting our hypotheses. 

In case of the constraint \eqref{0.1}, whenever for some 
$t_0 \le t_1 < t_2 \le \infty$ and all $t_1 < t < t_2$ there holds
$m(t)>C_0$ from \eqref{4.3a} and \eqref{0.3} we obtain
\begin{equation}\label{4.3d}
  \begin{split}
    t_2 - t_1 & \le C_0 \int_{0}^{\infty} 
    \bigg(\lambda(t)^{-1}\int_{\Omega} |u_t|^2 e^{u^2}dx\bigg) dt 
    \le C_0 E_{\infty} =: T_0 < \infty.
  \end{split}
\end{equation}
By \eqref{4.3d} the length of any interval $I=]t_1, t_2[$ with $m(t)>C_0$ for $t\in I$ 
is uniformly bounded. Since $\limsup_{t\rightarrow \infty} m(t) = \infty$, 
we may then assume that $m(t_1)=C_0$. 
Applying \eqref{2.5} to the shifted flow $u(t-t_1)$, by \eqref{1.4} for any such 
interval we find $\sup_{t_1 < t \le t_2} m(t) \le C_0 e^{C_2}$, where 
$C_2=2c_0^{-1}(\Lambda_0 T_0 + E_{\infty}) < \infty$. Thus we also have
$\limsup_{t \rightarrow \infty} m(t) \le C_0 e^{C_2}$, contrary to hypothesis.
\end{proof}

For a sequence $(t_k)$ as determined in Lemma \ref{lemma4.1} above we  
let $u_k = u(t_k)$, $k \in {\mathbb N}$ and set $\dot{u}_k = u_t(t_k)$.
The symbols $t$, $t_k$ then no longer explicitly appear and we
may use these letters for other purposes.
Also let 
$\eta = \log \left(\frac{2}{1+|x|^2}\right)$ be the standard solution of Liouville's 
equation
\begin{equation} \label{4.4}
   -\Delta \eta= e^{2\eta} \hbox{ on } {\mathbb R}^2
\end{equation}
induced by stereographic projection from $S^2$, with 
\begin{equation}\label{4.4a}
   \int_{{\mathbb R}^2}e^{2\eta}dx= 4 \pi =: \Lambda_1.
\end{equation}
Similar to \cite{Adimurthi-Struwe}, \cite{Druet} the following result now holds.

\begin{theorem}\label{thm4.1}
There exist a number $i_* \in {\mathbb N}$ and 
points $x^{(i)} \in \overline{\Omega}$, $1 \le i \le i_*$, such that  
as $k \rightarrow \infty$ suitably for each $i$ 
with suitable points $x_k= x^{(i)}_k \rightarrow x^{(i)}$ and scale factors 
$0 < r_k = r^{(i)}_k \rightarrow 0$ satisfying  
\begin{equation}\label{4.2}
  \lambda_k r_k^2 u_k^2(x_k) e^{u_k^2(x_k)} = 4
\end{equation}
we have 
\begin{equation}\label{4.3}
  \eta_k(x) = \eta^{(i)}_k(x) := 
  u_k(x_k)(u_k(x_k + r_k x ) - u_k(x_k)) 
  \rightarrow \eta_0 = \log \bigg(\frac{1}{1+|x|^2}\bigg)
\end{equation}
locally uniformly on ${\mathbb R}^2$, where 
$\eta_0 = \eta - \log 2$ satisfies
\begin{equation} \label{4.4b}
   -\Delta \eta_0= 4e^{2\eta_0} \hbox{ on } {\mathbb R}^2,
\end{equation}
and there holds
\begin{equation}\label{4.4c}
  \lim_{L \rightarrow\infty}\lim_{k \rightarrow\infty}
  \lambda_k \int_{B_{Lr_k}(x_k)} u_k^2 e^{u_k^2}dx 
  = 4\int_{{\mathbb R}^2}e^{2\eta_0}dx= \Lambda_1.
\end{equation}

Equality $x^{(i)}= x^{(j)}$ may occur, but we have
\begin{equation}\label{4.5}
    \frac{dist(x^{(i)}_k,\partial \Omega)}{r^{(i)}_k},\ 
    \frac{|x^{(i)}_k - x^{(j)}_k|}{ r^{(i)}_k}
    \to \infty \hbox{ for all } 1 \le i \neq j \le i_*,
\end{equation}
and there holds the uniform pointwise estimate
\begin{equation}\label{4.6}
     \lambda_k \inf_i |x-x^{(i)}_k|^2 u_k^2(x)e^{u_k^2(x)}\leq C,
\end{equation}
for all $x \in \Omega$ and all $k \in {\mathbb N}$.

Finally, $u_k\rightarrow u_{\infty}$ in 
$H^2_{loc}(\Omega \setminus \{x_1,\dots,x_{i_*}\})$ as $k \rightarrow\infty$.
\end{theorem}

\begin{proof}
Choose $x_k=x^{(1)}_k\in \Omega$ such that $u_k(x_k) = sup_{x\in \Omega}u_k$
and let $r_k = r^{(1)}_k$ be given by \eqref{4.2}. 
We claim that $r_k \rightarrow 0$ as $k \rightarrow\infty$.
Otherwise, \eqref{4.2} gives $\lambda_k u_k^2(x_k)e^{u_k^2(x_k)} \le C < \infty$, 
and with the help of Lemma \ref{lemma4.1} we can estimate
\begin{equation*}
   \int_{\Omega} |u_k(x_k)\dot{u}_k e^{u_k^2}|^2dx
   \le \lambda_k u_k^2(x_k) e^{u_k^2(x_k)}\bigg(\lambda_k^{-1}
   \int_{\Omega} \dot{u}_k^2 e^{u_k^2}\; dx\bigg)
   \rightarrow 0
\end{equation*}
as $k \rightarrow\infty$. By \eqref{1.1} then the sequence $(u_k(x_k)\Delta u_k)$ 
is bounded in $L^2$ and it follows that $u_k \rightarrow 0$ uniformly as 
$k \rightarrow\infty$ contradicting our assumption that $u_k(x_k)\rightarrow\infty$. 
Therefore $r_k \rightarrow 0$ as $k \rightarrow\infty$.

Suppose that we already have determined points $x^{(1)}_k,\dots, x^{(i-1)}_k$
such that \eqref{4.3} and \eqref{4.5} hold and let $x_k=x^{(i)}_k\in \Omega$ 
be such that 
\begin{equation}\label{4.6a}
     \lambda_k \inf_{j<i} |x_k-x^{(j)}_k|^2 u_k^2(x_k)e^{u_k^2(x_k)}
     = \sup_{x \in \Omega}\bigg( 
       \lambda_k \inf_{j<i} |x-x^{(j)}_k|^2 u_k^2(x)e^{u_k^2(x)}\bigg)
     \rightarrow\infty
\end{equation}
as $k \rightarrow\infty$. If no such $x_k=x^{(i)}_k$ exists the induction 
terminates, establishing \eqref{4.6}.

Choose $r_k = r^{(i)}_k \rightarrow 0$ satisfying \eqref{4.2}.
In view of \eqref{4.6a} we have 
$|x_k-x^{(j)}_k|/r_k \rightarrow \infty$ for all $j < i$; that is, 
half of \eqref{4.5}.
Moreover, denoting as
$v_k(x) = u_k(x_k+r_kx)$ the scaled function $u_k$ on the domain
\begin{equation*}
   \Omega_k = \{x; x_k + r_k x \in \Omega\},
\end{equation*}
with error $o(1) \rightarrow 0$ as $k \rightarrow \infty$ for any $L>0$
we can estimate
\begin{equation}\label{4.7}
     \sup_{x \in \Omega_k,\; |x| \le L} v_k^2(x)e^{v_k^2(x)}
     \le (1+o(1))v_k^2(0)e^{v_k^2(0)} = (1+o(1))u_k^2(x_k)e^{u_k^2(x_k)}.
\end{equation}

Let $\eta_k(x) = \eta^{(i)}_k(x)$ be defined as in \eqref{4.3}. Also denoting as
$\dot{v}_k(x) = \dot{u}_k(x_k+r_kx)$ the scaled function $\dot{u}_k=u_t(t_k)$, then we have
\begin{equation*}
   -  \Delta \eta_k = \lambda_k r_k^2 v_k(0)v_k e^{v_k^2} 
                          - r_k^2 \dot{v}_k v_k(0) e^{v_k^2} =:I_k + II_k 
   \hbox{ on }\Omega_k.
\end{equation*}
Observe that for any $L>0$ the bound \eqref{4.7} implies the uniform estimate
\begin{equation}\label{4.8}
 \begin{split}
   0 & < I_k= \lambda_k r_k^2 v_k(0)v_k e^{v_k^2} 
   \le \lambda_k r_k^2 \sup \{v_k^2(0)e^{v_k^2(0)},v_k^2 e^{v_k^2}\}\\
   & \le (1+o(1)) \lambda_k r_k^2 v_k^2(0)e^{v_k^2(0)} 
   = (4+o(1)) \hbox{ on } B_L(0)\; ;
 \end{split}
\end{equation}
moreover, with \eqref{4.2} and Lemma \ref{lemma4.1} 
for the second term we have 
\begin{equation}\label{4.9}
 \begin{split}
   \int_{\Omega_k \cap B_L(0)}|II_k|^2dx
   &\le (1+o(1))\lambda_k r_k^2 v_k^2(0) e^{v_k^2(0)}\bigg(\lambda_k^{-1}
    \int_{\Omega_k \cap B_L(0)}r_k^2 \dot{v}_k^2 e^{v_k^2}\; dx\bigg)\\
   &= (4+o(1))\lambda_k^{-1}
     \int_{\Omega \cap B_{Lr_k}(x_k)}|u_t(t_k)|^2 e^{u_k^2}\; dx
   \rightarrow 0
 \end{split}
\end{equation}
with error $o(1) \rightarrow 0$ as $k \rightarrow \infty$.

Note that \eqref{4.6a} forces $v_k(0) \rightarrow \infty$.
Since \eqref{4.7} also implies the bound 
\begin{equation}\label{4.10}
   2\eta_k = v_k^2-v_k^2(0)-(v_k-v_k(0))^2 \le o(1)
   \hbox{ on } \Omega_k \cap B_L(0)\; ,
\end{equation}
it follows that
\begin{equation*}
    dist(0,\partial \Omega_k) = \frac{dist(x_k,\partial \Omega)}{r_k}\to \infty .
\end{equation*}
Otherwise, by \eqref{4.8} - \eqref{4.10}, the mean value property 
of harmonic functions and the fact that $\eta_k \rightarrow -\infty$
on $\partial \Omega_k$ as $k \rightarrow \infty$ we have
locally uniform convergence $\eta_k \rightarrow -\infty$ in $\Omega_k$,
which contradicts the fact that $\eta_k(0)=0$. By the same reasoning we
also may assume that as $k \rightarrow \infty$
a subsequence $\eta_k \rightarrow \eta_{\infty}$ in $H^2_{loc}$ and
locally uniformly. Recalling that $v_k(0) \rightarrow \infty$, then we also
have
\begin{equation}\label{4.11}
   (v_k - v_k(0))\rightarrow 0,\ \rho_k := \frac{v_k}{v_k(0)}\rightarrow 1,\
   a_k := 1 + \frac{\eta_k}{2v_k^2(0)}\rightarrow 1
\end{equation}
locally uniformly.
Observing that $e^{v_k^2-v_k^2(0)} = e^{2a_k \eta_k}$ and using \eqref{4.2},
we conclude
\begin{equation*}
   I_k = \lambda_k r_k^2 v_k(0)v_k e^{v_k^2} = 4 \rho_k e^{2a_k \eta_k}
   \rightarrow 4 e^{2\eta_{\infty}}
\end{equation*}
locally uniformly. Thus, $\eta_{\infty}$ solves \eqref{4.4b}; 
moreover, for any $L > 1$ by \eqref{2.2} or \eqref{0.5} we have 
\begin{equation*}
 \begin{split}
   4 \int_{B_L(0)}e^{2 \eta_{\infty}} \; dx 
   & = \lim_{k \rightarrow \infty}\int_{B_{L}(0)}4\rho_k^2 e^{2a_k \eta_k}dx
   = \lim_{k \rightarrow \infty}\int_{B_{Lr_k}(x_k)}\lambda_k u_k^2 e^{u_k^2} \; dx
   \le \Lambda_0.
  \end{split}
\end{equation*}
By Fatou's lemma, upon letting $L \rightarrow \infty$ we find 
$\int_{{\mathbb R}^2}e^{2 \eta_{\infty}}dx < \infty$.
In view of the equation $\eta(0) = \lim_{k \rightarrow \infty}\eta_k(0) = 0$ 
together with \eqref{4.10}, the classification of Chen-Li \cite{Chen-Li} 
then yields that $\eta_{\infty} = \eta - \log 2 = \eta_0$, as claimed,
which completes the induction step. In view of \eqref{4.4c} the 
induction must terminate when $i > \Lambda_0/\Lambda_1$.

Finally, to see the asserted local $H^2$-convergence away from $x_i$, 
$1\le i \le i_*$, observe that by \eqref{4.6} and estimates similar to
\eqref{4.8}, \eqref{4.9} for any $x_0$ with 
\begin{equation*}
   \inf_{1 \le i \le i_*} |x_0 - x_k^{(i)}| \ge 3R_0 >0
\end{equation*} 
the sequence $(\Delta u_k)$ is bounded in $L^2(B_{2R_0}(x_0))$. 
Boundedness of $(u_k)$ on $B_{R_0}(x_0)$ and convergence 
$u_k \rightarrow u_{\infty}$ in $H^2(B_{R_0}(x_0))$ 
then follow from boundedness of $(E(u_k))$ and elliptic regularity.
\end{proof}

\section{Quantization}
Throughout this section we continue to assume that 
$\limsup_{t\rightarrow \infty}||u(t)||_{L^{\infty}} = \infty$ and
for a sequence $(t_k)$ as determined in Lemma \ref{lemma4.1} we let 
$u_k=u(t_k)\rightharpoondown u_{\infty}$ 
weakly in $H^1_0(\Omega)$ as $k \rightarrow \infty$, and
$\dot{u}_k=u_t(t_k)$ as above. 
By \eqref{2.2} or \eqref{0.5}, respectively, with error $o(1)\rightarrow 0$ 
there holds
\begin{equation} \label{5.1}
   \int_{\Omega} |\nabla u_k|^2 \; dx = (1+o(1))
   \lambda_k \int_{\Omega} u_k^2 e^{2u_k^2}\; dx 
   \rightarrow \Lambda
\end{equation}
for some $\Lambda < \infty$. By Theorem \ref{thm4.1}, moreover, we may assume
that 
\begin{equation*}
  |\nabla u_k|^2dx \overset{w^*}{\rightharpoondown} |\nabla u_{\infty}|^2dx 
  + \sum_{i=1}^{i_*} L^{(i)} \delta_{x^{(i)}}
\end{equation*}
and similarly 
\begin{equation*}
  \lambda_k u_k^2 e^{u_k^2} \overset{w^*}{\rightharpoondown} 
  \lambda_{\infty} u_{\infty}^2 e^{u_{\infty}^2} 
  + \sum_{i=1}^{i_*} \Lambda^{(i)} \delta_{x^{(i)}};
\end{equation*}
weakly in the sense of measures, where $\Lambda^{(i)} \ge \Lambda_1= 4\pi$
on account of \eqref{4.4c}. In fact, we have $L^{(i)}=\Lambda^{(i)}$, as
may be seen from the equations 
\begin{equation*}
  |\nabla u_k|^2 - \Delta(u_k^2/2) =\lambda_k u_k^2 e^{u_k^2} - \dot{u}_ku_ke^{u_k^2}
\end{equation*}
and 
\begin{equation*}
  |\nabla u_{\infty}|^2 - \Delta(u_{\infty}^2/2) 
  =\lambda_{\infty}u_{\infty}^2 e^{u_{\infty}^2}
\end{equation*}
that we obtain upon multiplying the equations \eqref{1.1}, \eqref{1.01} for $u_k$ 
and $u_{\infty}$ by the functions $u_k$ and $u_{\infty}$, 
respectively, together with the estimate \eqref{0.7}
that results from \eqref{5.1} and Lemma \ref{lemma4.1}. Finally, we use
convergence 
$$
   \int_{\Omega}\Delta(u_k^2-u_{\infty}^2)\varphi\; dx =
   \int_{\Omega}(u_k^2-u_{\infty}^2)\Delta\varphi\;dx\rightarrow 0 \quad (k \rightarrow \infty)
$$
for any $\varphi \in C^{\infty}(\overline{\Omega})$ and observe that this set of testing 
functions allows to separate point masses concentrated at points 
$x^{(i)}\in \overline{\Omega}$ to conclude.

Similar to \cite{Druet} and \cite{Struwe07} we then obtain the
following quantization result for the ``defect'' $\Lambda^{(i)}$ at each $x^{(i)}$.

\begin{theorem} \label{thm5.1}
We have $\Lambda^{(i)}= 4\pi l_i=l_i \Lambda_1$ for some $l_i \in {\mathbb N}$, 
$1 \le i \le i_*$.
\end{theorem}

For the proof we argue as in \cite{Struwe07}. We first consider the
radial case. 

\subsection{The radial case}
Let $\Omega = B_R(0)=: B_R$
and assume that $u(t,x) = u(t,|x|)$. 
In this case by Theorem \ref{thm4.1} for any $i \le i_*$ we have 
$r_k^{-1} x_k \rightarrow 0$ as $k \rightarrow \infty$, where 
$x_k = x^{(i)}_k$ and $r_k = r^{(i)}_k$ is given by \eqref{4.2};
otherwise, the blow-up limit 
$\eta_0 = \lim_{k \rightarrow \infty}\eta^{(i)}_k$ 
could not be radially symmetric. 
In particular, from \eqref{4.5} it follows that $i_* =1$; 
moreover, by \eqref{4.3} we have 
$u_k^2(x_k) = \sup_{\Omega} u_k^2 = u_k^2(0)+ o(1)$.
Thus, up to an error $o(1) \rightarrow 0$ locally uniformly 
as $k \rightarrow \infty$
we may replace the original function $\eta_k=\eta^{(1)}_k$ 
defined in \eqref{4.3} by the function 
\begin{equation*}
      \eta_k(x) = u_k(0)(u_k(r_k x) - u_k(0)).
\end{equation*}
 
Observe that by radial symmetry or Theorem \ref{thm4.1} we also have 
convergence $u_k \rightarrow u_{\infty}$ locally uniformly 
away from $x = 0$ as $k \rightarrow \infty$.

For $|x| = r$ let $u_k(r) = u_k(x)$ and set 
\begin{equation*}
   \lambda_k u_k^2 e^{u_k^2} =: e_k \hbox{ in } \Omega\, .
\end{equation*}
We also denote as 
\begin{equation*}
      w_k(x) = u_k(0)(u_k(x) - u_k(0))
\end{equation*}
the unscaled function $\eta_k$, satisfying the equation
\begin{equation*}
   - \Delta w_k = \lambda_k u_k(0) u_k e^{u_k^2} - d_k,
\end{equation*}
where the term $d_k = u_k(0) \dot{u}_k e^{u_k^2}$ for any $L > 0$ 
can be estimated
\begin{equation} \label{5.2}
 \begin{split}
   & \int_{B_{Lr_k}} |d_k|\; dx\\
   & \quad \le \sup_{B_{Lr_k}}\bigg(\frac{u_k(0)}{u_k}\bigg)
   \bigg(\lambda_k \int_{B_{Lr_k}} u_k^2 e^{u_k^2}dx \cdot
   \lambda_k^{-1} \int_{B_{Lr_k}}\dot{u}_k^2 e^{u_k^2}\; dx\bigg)^{1/2}.
 \end{split}
\end{equation}
Hence by Theorem \ref{thm4.1}, Lemma \ref{lemma4.1}, and \eqref{5.1} 
we conclude that 
$d_k\rightarrow 0$ in $L^1(B_{Lr_k})$ for any $L > 0$ as 
$k \rightarrow \infty$. Finally, we set 
\begin{equation*}
   \lambda_k u_k(0) u_k e^{u_k^2} =:f_k \hbox{ in } \Omega = B_R
\end{equation*}
and for $0 < r < R$ let 
\begin{equation*}
   \Lambda_k(r) = \int_{B_r} e_k \; dx, \;
   \sigma_k(r) = \int_{B_r} f_k \; dx ,\; 
\end{equation*}
Observe that with error $o(1) \rightarrow 0$ as $k \rightarrow \infty$ we have
$e_k \le (1+o(1))f_k $, $\Lambda_k(r)\le \sigma_k(r)+o(1)$; 
moreover, Theorem \ref{thm4.1} implies
\begin{equation} \label{5.3}
   \lim_{L \rightarrow \infty} \lim_{k \rightarrow \infty} \Lambda_k(Lr_k) 
   = \lim_{L \rightarrow \infty} \lim_{k \rightarrow \infty} \sigma_k(Lr_k)
   = \lim_{L \rightarrow \infty}  4 \int_{B_L} e^{2\eta_0}\; dx 
   = \Lambda_1.
\end{equation}

We can now show our first decay estimate. 
Let $u_k' = \frac{\partial u_k}{\partial r}$, and so on.

\begin{lemma}\label{lemma5.1}
For any $0 < \varepsilon < 1$, letting $T_k > 0$ be minimal such that 
$u_k(T_k) = \varepsilon u_k(0)$, for any constant $b < 2$ and sufficiently 
large $k$ there holds 
\begin{equation*}
   w_k(r) \le b \log\left(\frac{r_k}{r}\right) \hbox{ on } B_{T_k}
\end{equation*}
and we have
\begin{equation*}
   \lim_{k \rightarrow \infty}\Lambda_k(T_k) = 
   \lim_{k \rightarrow \infty}\sigma_k(T_k) = \Lambda_1 = 4\pi.
\end{equation*}
\end{lemma}

\begin{proof} Note that $T_k \rightarrow 0$ as $k \rightarrow \infty$
in view of the locally uniform convergence $u_k \rightarrow u_{\infty}$ 
away from $0$.

Since $u_k(t) \ge \varepsilon u_k(0)$ for $Lr_k \le t \le T_k$, from 
\eqref {5.3} and an estimate similar to \eqref {5.2} for all such $t=t_k$ we obtain 
\begin{equation}\label{5.4}
  \begin{split} 
     2\pi t w_k'(t)
     & = \int_{\partial B_t}\partial_{\nu} w_k  \; do
     = \int_{B_t}\Delta w_k  \; dx\\
     & = - \sigma_k(t) + o(1) \le - \Lambda_1 + o(1)
  \end{split}
\end{equation}
with error $o(1) \rightarrow 0$ uniformly in $t$, if first $k \rightarrow \infty$ 
and then $L \rightarrow \infty$. For any $b < 2$ and sufficiently large $L \ge L(b)$,
for $k \ge k_0(L)$ we thus obtain that
\begin{equation*}
  w_k'(t) \le - \frac{b}{t} 
\end{equation*}
for all $Lr_k \le t \le T_k$. Since $\eta_0(L) < - b \log L$ for all $L > 0$,
in view of Theorem \ref{thm4.1} clearly we may choose $k_0(L)$ such that 
$\eta_k(L) < - b \log L$ for all $k \ge k_0(L)$. For any such $k$ and any 
$r \in [Lr_k,T_k]$, upon integrating from $Lr_k$ to $r$ then we find
\begin{equation}\label{5.5}
  \begin{split}
  w_k(r) & \le w_k(Lr_k) - b \log \left(\frac{r}{Lr_k}\right)\\
  & = \eta_k(L) + b \log L + b \log \left(\frac{r_k}{r}\right)
  \le b \log \left(\frac{r_k}{r}\right),
  \end{split}
\end{equation}
as claimed. For $r \le Lr_k$ the asserted bound already follows 
from Theorem \ref{thm4.1}.

Inserting \eqref{5.5} in the definition
of $f_k$ and recalling \eqref{4.2}, for $Lr_k \le r \le T_k$ 
with sufficiently large $L > 0$ and $k \ge k_0(L)$ then we obtain 
\begin{equation*}
 \begin{split}
  f_k 
  &= \lambda_k (u_k^2(0)+w_k) e^{u_k^2(0)} e^{2(1 + \frac{w_k}{2u_k^2(0)}) w_k}\\
  &\le \lambda_k r_k^2 u_k^2(0) e^{u_k^2(0)} r_k^{-2} e^{(1 + \varepsilon) w_k}
  \le 4 r_k^{-2}  \left(\frac{r_k}{r}\right)^{(1 + \varepsilon)b}.
 \end{split}
\end{equation*}
Choosing $b < 2$ such that $(1 + \varepsilon)b = 2 + \varepsilon$, upon
integrating over $B_{T_k}$ we obtain
\begin{equation*}
 \begin{split}
   \sigma_k({T_k}) & = \int_{B_{T_k}} f_k \; dx 
   \le \Lambda_1 + \int_{B_{T_k} \setminus B_{Lr_k}} f_k \; dx \\
   &\le \Lambda_1 + C  r_k^{-2} \int_{B_{T_k} \setminus B_{Lr_k}}
   \left(\frac{r_k}{r}\right)^{2 + \varepsilon} \; dx 
   \le \Lambda_1 + C \varepsilon^{-1}
   \left(\frac{r_k}{Lr_k}\right)^{\varepsilon}
   \le \Lambda_1 + \varepsilon,
 \end{split}
\end{equation*}
if first $L > L_0(\varepsilon)$ and then $k \ge k_0(L)$ 
is chosen sufficiently large. Since $\varepsilon > 0$ is arbitrary, the proof 
is complete.
\end{proof}

If we now choose $\varepsilon_k \downarrow 0$ such that with 
$s_k = T_k(\varepsilon_k)$ we have $u_k(s_k)\rightarrow \infty$, 
by Theorem \ref{thm4.1} we also have $r_k/s_k \rightarrow 0$, $s_k \rightarrow 0$ 
as $k \rightarrow \infty$. That is, we can achieve that 
\begin{equation}\label{5.6}
   \lim_{k \rightarrow \infty} \Lambda_k(s_k) =\Lambda_1, \; 
   \lim_{k \rightarrow \infty}\frac{u_k(s_k)}{u_k(r_k)} 
   = \lim_{k \rightarrow \infty}\frac{r_k}{s_k}
   = \lim_{k \rightarrow \infty} s_k = 0.
\end{equation}
In addition, from \eqref{5.3} we obtain that 
\begin{equation}\label{5.7}
   \lim_{L \rightarrow \infty}\lim_{k \rightarrow \infty} 
   (\Lambda_k(s_k) - \Lambda_k(Lr_k)) = 0.
\end{equation}

Let $r_k=r_k^{(1)}$, $s_k=s_k^{(1)}$. 
We now proceed by iteration. Suppose that for 
some integer $l \ge 1$ we already have determined numbers 
$r_k^{(1)} < s_k^{(1)} < \dots < r_k^{(l)} < s_k^{(l)}$ such that
\begin{equation}\label{5.8}
   \lim_{k \rightarrow \infty} \Lambda_k(s_k^{(l)}) = l \Lambda_1
\end{equation}
and
\begin{equation}\label{5.9}
   \lim_{L \rightarrow \infty}\lim_{k \rightarrow \infty} 
      (\Lambda_k(s_k^{(l)})-\Lambda_k(Lr_k^{(l)})) 
   = \lim_{k \rightarrow \infty}\frac{u_k(s_k^{(l)})}{u_k(r_k^{(l)})}
   = \lim_{k \rightarrow \infty}\frac{r_k^{(l)}}{s_k^{(l)}} 
   = \lim_{k \rightarrow \infty} s_k^{(l)} = 0.
\end{equation}
For $0 < s < t < R$ let 
\begin{equation*}
   N_k(s,t) =\int_{B_t \setminus B_s} e_k \; dx 
   = \int _{B_t \setminus B_s} \lambda_k u_k^2 e^{u_k^2}dx 
   = 2\pi \int_s^t \lambda_k r u_k^2 e^{u_k^2}dr
\end{equation*}
and define
\begin{equation*}
   P_k(t) = t \frac{\partial}{\partial t}N_k(s,t) = 
   t \int_{\partial B_t} e_k \; do 
   = 2\pi  \lambda_k t^2 u_k^2(t) e^{u_k^2(t)}.
\end{equation*}
Note that \eqref{4.6} implies the uniform bound $P_k \le C$; moreover,
with a uniform constant $C_0$ for any $t$ we have
\begin{equation}\label{5.10}
  \begin{split} 
   \inf_{t/2 \le t' \le t}P_k(t') \le C_0 N_k(t/2,t).
  \end{split} 
\end{equation}

A preliminary quantization now can be achieved, as follows.

\begin{lemma}\label{lemma5.2}
i) Suppose that for some $t_k > s_k^{(l)}$ there holds
\begin{equation*}
   \sup_{s_k^{(l)} <t<t_k} P_k(t) \rightarrow 0 \hbox{ as } k \rightarrow \infty.
\end{equation*}
Then we have
\begin{equation*}
   \lim_{k \rightarrow \infty}N_k(s_k^{(l)} ,t_k) = 0.
\end{equation*}
ii) Conversely, if for some $t_k > s_k^{(l)}$ and a subsequence $(u_k)$ there holds
\begin{equation*}
   \lim_{k \rightarrow \infty}N_k(s_k^{(l)} ,t_k) = \nu_0 > 0,   \
   \lim_{k\rightarrow \infty} t_k = 0,
\end{equation*}
then either $\nu_0 \ge \pi$, or we have 
\begin{equation*}
   \liminf_{k \rightarrow \infty}P_k(t_k) \ge \nu_0
\end{equation*}
and
\begin{equation*}
   \lim_{L \rightarrow \infty}\liminf_{k \rightarrow \infty}N_k(s_k^{(l)} ,L t_k) 
   \ge \pi, \;
   \lim_{L \rightarrow \infty}\limsup_{k \rightarrow \infty} N_k(s_k^{(l)} ,t_k/L) = 0.
\end{equation*}
\end{lemma}

\begin{proof}
{\it i)} For $s=s_k^{(l)} < t$ we integrate by parts to obtain 
\begin{equation} \label{5.11}
  \begin{split} 
   2 N_k(s,t) & = \int _{B_t \setminus B_s} e_k \;div\; x\; dx = 
   P_k(t) - P_k(s) - \int _{B_t \setminus B_s} x\cdot \nabla e_k \; dx \\
   & \le P_k(t)
   - 4 \pi \int_s^t \lambda_k r^2 u_k'(1 + u_k^2)u_k e^{u_k^2}dr.
     \end{split} 
\end{equation}

In order to further estimate the right hand side we observe that \eqref {1.1} 
for any $t < R$ yields the identity 
\begin{equation} \label{5.11a}
   - 2\pi t u_k(t) u_k'(t) = \int_{B_t} \lambda_k u_k(t) u_k e^{u_k^2} dx 
     - \int_{B_t} u_k(t) \dot{u}_k e^{u_k^2} dx.
\end{equation}
Estimating $u_k^2(t) e^{u_k^2} \le \max\{u_k^2(t) e^{u_k^2(t)},u_k^2 e^{u_k^2}\}$,
by Lemma \ref{lemma4.1}, \eqref{4.6}, and \eqref{5.1}, 
we can easily bound the contribution from the second integral 
\begin{equation} \label{5.11b}
  \begin{split} 
     & \big(\int_{B_t}u_k(t) |\dot{u}_k| e^{u_k^2} dx \big)^2
     \le \lambda_k \int_{B_t} u_k^2(t) e^{u_k^2}dx \cdot
     \lambda_k^{-1} \int_{B_t}\dot{u}_k^2 e^{u_k^2} dx\\
     & \quad \quad \le o(1) \big(\pi\lambda_k t^2 u_k^2(t) e^{u_k^2(t)} 
            + \lambda_k \int_{B_t} u_k^2 e^{u_k^2} dx \big) = o(1),
  \end{split} 
\end{equation}
where $o(1) \rightarrow 0$ as $k \rightarrow \infty$. From \eqref {5.11a} 
we then obtain that at any sequence of points $t=t_k$ where $u_k'(t) \ge 0$
there holds 
\begin{equation}\label{5.11c}
   \int_{B_t} \lambda_k u_k(t) u_k e^{u_k^2} dx = o(1).
\end{equation}
On the other hand, if for $t_{k0}=t_0 \le r \le t=t_k$ there holds
$u_k'(r) \le 0 = u_k'(t_0)$, by \eqref {5.11c} we can estimate
\begin{equation} \label{5.11d}
  \begin{split} 
   \int_{B_t} \lambda_k u_k(t) u_k e^{u_k^2} dx 
   & \le \int_{B_t\setminus B_{t_0}} \lambda_k u_k^2 e^{u_k^2} dx
   + \int_{B_{t_0}} \lambda_k u_k(t_0) u_k e^{u_k^2} dx \\
   & = N_k(t_0,t) + o(1).
  \end{split} 
\end{equation}
In view of \eqref {5.11b}-\eqref {5.11d} and \eqref {5.9}, 
for $s=s_k^{(l)} \le r \le t=t_k$ and with $r_k=r_k^{(l)}$ we then can estimate 
\begin{equation} \label{5.11e}
  \begin{split} 
   - 2\pi r u_k(r) u_k'(r) 
   & = \int_{B_r} \lambda_k u_k(r) u_k e^{u_k^2} dx + o(1) \\
   & \le N_k(s,r) + \int_{B_s} \lambda_k u_k(s) u_k e^{u_k^2} dx + o(1)\\
   & \le N_k(s,r) + N_k(Lr_k,s) + \frac{u_k(s)}{u_k(Lr_k)}\Lambda_k(Lr_k)+ o(1)\\
   & = N_k(s,r)+ o(1),
  \end{split}
\end{equation}
where $o(1) \rightarrow 0$ when first $k \rightarrow \infty$ and then 
$L \rightarrow \infty$. Indeed, the first inequality is clear when $u_k' \le 0$
in $[s,r]$, and otherwise follows from \eqref {5.11c}, \eqref {5.11d}. The second inequality
may be seen in a similar way. Recalling \eqref{5.11} we thus arrive at the estimate
\begin{equation} \label{5.12}
  \begin{split} 
   2 N_k(s,t) & \le P_k(t) + 2 \int_s^t \lambda_k r (1 + u_k^2)e^{u_k^2}N_k(s,r)dr+ o(1)\\
   & \le P_k(t) + \pi^{-1}N_k(s,t)^2 + o(1).
     \end{split} 
\end{equation}

If we now assume that
\begin{equation*}
   \sup_{s<t<t_k} P_k(t) \rightarrow 0 \hbox{ as } k \rightarrow \infty,
\end{equation*}
upon letting $t$ increase from $t=s=s_k^{(l)}$ to $t_k$ we find
\begin{equation*}
   \lim_{k \rightarrow \infty}N_k(s_k^{(l)} ,t_k) = 0,
\end{equation*}
as claimed. 

ii) On the other hand, if we suppose that
for some $t_k > s_k^{(l)}$ we have
\begin{equation} \label{5.13}
   0 < \lim_{k \rightarrow \infty}N_k(s_k^{(l)} ,t_k) = \nu_0 < \pi,
\end{equation}
from \eqref{5.12} with error $o(1) \rightarrow 0$ as $k \rightarrow \infty$
we conclude that 
\begin{equation}\label{5.14}
    \nu_0 + o(1) \le (2 - \nu_0/\pi) N_k(s_k^{(l)} ,t_k) \le P_k(t_k) + o(1).
\end{equation}
It then also follows that 
\begin{equation*}
   \lim_{L \rightarrow \infty}\liminf_{k \rightarrow \infty}N_k(s_k^{(l)} ,Lt_k) 
   \ge \pi.
\end{equation*}
Otherwise, \eqref{5.10} and \eqref{5.14} for a subsequence $(u_k)$ 
yield the uniform bound 
\begin{equation*}
    C_0 \liminf_{k \rightarrow \infty}N_k(Lt_k/2 ,Lt_k) 
    \ge \liminf_{k \rightarrow \infty}\inf_{Lt_k/2\le t \le Lt_k}P_k(t) \ge \nu_0
\end{equation*}
for all $L \ge 2$. Choosing $L = 2^m$, where $m \in {\mathbb N}$, 
and summing over $1 \le m \le M$, we obtain
\begin{equation*}
    C_0 \liminf_{k \rightarrow \infty}\Lambda_k(2^M t_k) \ge 
    C_0 \liminf_{k \rightarrow \infty}N_k(t_k ,2^M t_k) \ge \nu_0 M  
    \rightarrow \infty \hbox{ as } M \rightarrow \infty,
\end{equation*}
contrary to assumption \eqref{5.1}. 
Upon replacing $t_k$ by $t_k/L$ in the previous argument and recalling our
assumption \eqref{5.13}, by the same reasoning we also arrive at the estimate
\begin{equation*}
    \lim_{L \rightarrow \infty}\liminf_{k \rightarrow \infty}N_k(s_k^{(l)} ,t_k/L) = 0.
\end{equation*}
This completes the proof.
\end{proof}

Suppose that for some $t_k > s_k^{(l)}$ with $t_k \rightarrow 0$ as 
$k \rightarrow \infty$ there holds
\begin{equation*}
   \liminf_{k \rightarrow \infty}N_k(s_k^{(l)} ,t_k) > 0.
\end{equation*}
Then we can find a subsequence $(u_k)$ and numbers 
$r_k^{(l+1)} \in ]s_k^{(l)},t_k[$ such that 
\begin{equation} \label{5.15}
   \lim_{k \rightarrow \infty}N_k(s_k^{(l)},r_k^{(l+1)}) =  \nu_0 > 0 .
\end{equation}
Replacing our original choice of $r_k^{(l+1)}$ by a smaller number, 
if necessary, we may assume that $\nu_0 < \pi$.
Lemma \ref{lemma5.2} then implies that
\begin{equation}\label{5.16}
   \lim_{L \rightarrow \infty}\liminf_{k \rightarrow \infty}
   N_k(s_k^{(l)},Lr_k^{(l+1)}) \ge \pi, \;
   \lim_{L \rightarrow \infty}\limsup_{k \rightarrow \infty}
   N_k(s_k^{(l)},r_k^{(l+1)}/L) = 0,
\end{equation}
and that 
\begin{equation} \label{5.17}
    \liminf_{k \rightarrow \infty} P_k(r_k^{(l+1)}) > 0.
\end{equation}
In particular, since $r_k^{(l+1)} \le t_k \rightarrow 0$ we then conclude that
$u_k(r_k^{(l+1)}) \rightarrow \infty$.

The desired precise quantization result at the scale $r_k^{(l+1)}$ is a
consequence of the following Proposition.

\begin{proposition} \label{prop5.1} 
There exist a subsequence $(u_k)$ such that
\begin{equation*}
  \eta^{(l+1)}_k(x) := 
  u_k(r_k^{(l+1)})(u_k(r_k^{(l+1)} x ) - u_k(r_k^{(l+1)})) 
  \rightarrow \eta(x)
\end{equation*}
locally uniformly on ${\mathbb R}^2\setminus \{0\}$ as $k \rightarrow \infty$, 
where $\eta(x) = \log(\frac{2}{1+|x|^2})$.
\end{proposition}

Postponing the details of the proof of Proposition \ref{prop5.1} to the next section,
we now complete the proof of Theorem \ref{thm5.1}. 

Denote as $v_k^{(l+1)}(x) = u_k(r_k^{(l+1)} x)$, 
$\dot{v}_k^{(l+1)}(x) = \dot{u}_k(r_k^{(l+1)}x)$ 
the scaled functions $u_k$ and $\dot{u}_k$, respectively. Omitting the
superscript $(l+1)$ for brevity, similar to the proof of Theorem \ref{thm4.1}
for $\eta_k:=\eta_k^{(l+1)}$ we have
\begin{equation*}
   -  \Delta \eta_k = \lambda_k r_k^2 v_k(1) v_k e^{v_k^2} 
   - r_k^2 \dot{v}_k v_k(1) e^{v_k^2} =:I_k + II_k,
\end{equation*}
where $II_k \rightarrow 0$ in $L^2_{loc}({\mathbb R}^2\setminus \{0\})$ 
as $k \rightarrow \infty$.
Moreover, letting $\rho_k=\rho^{(l+1)}_k := \frac{v_k}{v_k(1)}$, 
$a_k=a^{(l+1)}_k = 1 + \frac{\eta_k}{2v^2_k(1)}$, by 
Proposition \ref{prop5.1} we have 
$a_k \rightarrow  1$, $\rho_k \rightarrow 1$ as $k \rightarrow \infty$
locally uniformly away from $x=0$, and
\begin{equation*}
   I_k = \lambda_k r_k^2 v_k(1) v_k e^{v_k^2} 
         = \lambda_k r_k^2 v^2_k(1) e^{v_k^2(1)}\rho_k e^{v_k^2-v_k^2(1)}
         = (2\pi)^{-1}P_k(r_k)\rho_k e^{2a_k \eta_k}.
\end{equation*}

Now observe that $\eta$ solves equation \eqref{4.4} on ${\mathbb R}^2$ with 
\begin{equation*}
  \int_{{\mathbb R}^2}  e^{2\eta} dx = 4\pi = \Lambda_1.
\end{equation*}
We therefore conclude that $P_k(r_k) \rightarrow 2\pi$ and 
\begin{equation} \label{5.18}
 \begin{split}
   \lim_{L \rightarrow \infty} & \lim_{k \rightarrow \infty}
   N_k(r_k^{(l+1)}/L,Lr_k^{(l+1)})
   =  \lim_{L \rightarrow \infty} \lim_{k \rightarrow \infty}
   \int_{B_L\setminus B_{1/L}}\lambda_k r_k^2 v^2_k e^{v_k^2}dx\\
   = & \lim_{L \rightarrow \infty} \lim_{k \rightarrow \infty}
   \int_{B_L\setminus B_{1/L}}(2\pi)^{-1}P_k(r_k)\rho_k^2 e^{2a_k \eta_k}dx
   = \lim_{L \rightarrow \infty} \int_{B_L\setminus B_{1/L}}e^{2\eta}dx
   = \Lambda_1 .
  \end{split}
\end{equation}

From \eqref{5.16} then we obtain that 
\begin{equation*}
  \begin{split}
   \lim_{L \rightarrow \infty} & \lim_{k \rightarrow \infty} N_k(s_k^{(l)},Lr_k^{(l+1)})\\ 
   & = \lim_{L \rightarrow \infty} \lim_{k \rightarrow \infty}
   (N_k(s_k^{(l)},r_k^{(l+1)}/L) + N_k(r_k^{(l+1)}/L,Lr_k^{(l+1)})) 
   = \Lambda_1,
  \end{split}
\end{equation*}
and our induction hypothesis \eqref{5.8} yields
\begin{equation} \label{5.19}
   \lim_{L \rightarrow \infty} \lim_{k \rightarrow \infty} \Lambda_k(Lr_k^{(l+1)}) 
   = \lim_{L \rightarrow \infty} \lim_{k \rightarrow \infty} 
   (\Lambda_k(s_k^{(l)}) + N_k(s_k^{(l)},Lr_k^{(l+1)})) = (l+1) \Lambda_1.
\end{equation}

Moreover, $r_k^{(l+1)}/s_k^{(l)} \rightarrow \infty$ as $k \rightarrow \infty$.
Indeed, if we assume that $r_k^{(l+1)} \le Ls_k^{(l)}$ for some $L$ 
by Proposition \ref{prop5.1} we have
$N_k(s_k^{(l)}/2,s_k^{(l)}) \ge \nu_0$ for some constant $\nu_0 = \nu_0(L) > 0$,
contradicting \eqref{5.9}. 

In order to obtain decay analogous to Lemma \ref{lemma5.1} and then also 
the analogue of \eqref{5.9} at the scale $r_k^{(l+1)}$, denote as 
\begin{equation*}
      w_k^{(l+1)}(x) = u_k(r_k^{(l+1)})(u_k(x) - u_k(r_k^{(l+1)}))
\end{equation*}
the unscaled function $\eta_k^{(l+1)}$, satisfying the equation
\begin{equation} \label{5.20}
  \begin{split}
   -\Delta w_k^{(l+1)} & = \lambda_k u_k(r_k^{(l+1)}) u_k e^{u_k^2}
   - u_k(r_k^{(l+1)}) \dot{u}_k e^{u_k^2} 
   =: f_k^{(l+1)} - d_k^{(l+1)}
  \end{split}
\end{equation}
in $\Omega = B_R$. We then have the analogue of Lemma \ref{lemma5.1},
which may be proved in the same fashion.

\begin{lemma}\label{lemma5.3}
For any $\varepsilon >0$, letting $T_k =T_k^{(l+1)}> r_k^{(l+1)}$ be minimal
such that $u_k(T_k) = \varepsilon u_k(r_k^{(l+1)})$, for any constant $b < 2$
and sufficiently large $k$ and $L$ there holds 
\begin{equation*}
   w_k^{(l+1)}(r) \le b \log\left(\frac{r_k^{(l+1)}}{r}\right)
   \hbox{ on } B_{T_k} \setminus B_{L r_k^{(l+1)}}
\end{equation*}
and we have
\begin{equation*}
   \lim_{k \rightarrow \infty}N_k(s_k^{(l)},T_k) = \Lambda_1.
\end{equation*}
\end{lemma}

\begin{proof}
Denote $w_k^{(l+1)}=w_k$, $r_k^{(l+1)}=r_k$, $d_k^{(l+1)}=d_k$ for simplicity. 
Coupled with the uniform bound $u_k(t) \ge \varepsilon u_k(r_k)$ for 
$r_k \le t \le T_k$, the estimate \eqref{5.11b} yields decay of 
$\int_{B_{T_k}}|d_k| dx$. Thus, for $L r_k \le t=t_k \le T_k$ 
from \eqref{5.18} and Proposition \ref{prop5.1} we have 
\begin{equation}\label{5.21}
  \begin{split} 
     2\pi t w_k'(t)
     & = \int_{\partial B_t}\partial_{\nu} w_k  \; do
     = \int_{B_t}\Delta w_k  \; dx 
        \le - \int_{B_{L r_k}} \frac{u_k(r_k)}{u_k} e_k  \; dx + o(1)\\
     & \le - N_k(r_k/L,Lr_k) + o(1)
        \le - \Lambda_1 + o(1),
  \end{split}
\end{equation}
with error $o(1) \rightarrow 0$ uniformly in $t$, if first $k \rightarrow \infty$ and 
then $L \rightarrow \infty$. For any $b < 2$ and sufficiently large $L \ge L(b)$
for $k \ge k_0(L)$, we thus obtain that 
\begin{equation*}
   w_k'(t) \le - \frac{b}{t} 
\end{equation*}
for all $L r_k \le t \le T_k$.
For such $t$ it then follows that 
\begin{equation*}
 \begin{split}
  e_k & \le \lambda_k u_k^2(r_k) e^{u_k^2(r_k)} 
            e^{2(1 + \frac{w_k}{2u_k^2(r_k)}) w_k}\\
  & \le (2\pi)^{-1}P(r_k) r_k^{-2} e^{(1 + \varepsilon) w_k}
  \le C r_k^{-2}\left(\frac{r_k}{r}\right)^{(1 + \varepsilon)b},
 \end{split}
\end{equation*}
and the proof may be completed as in Lemma \ref{lemma5.1}.
\end{proof}

For suitable numbers $s_k^{(l+1)}=T_k^{(l+1)}(\varepsilon_k)$,
where $\varepsilon_k \downarrow 0$ is chosen such that 
$u_k(s_k^{(l+1)}) = \varepsilon_k u_k(r_k^{(l+1)})\rightarrow \infty$
as $k \rightarrow \infty$, then we have 
\begin{equation}\label{5.22}
   \lim_{k \rightarrow \infty} \Lambda_k(s_k^{(l+1)}) = (l+1) \Lambda_1
\end{equation}
and
\begin{equation}\label{5.23}
 \begin{split}
   \lim_{L \rightarrow \infty} & \lim_{k \rightarrow \infty}
     (\Lambda_k(s_k^{(l+1)}) - \Lambda_k(Lr_k^{(l+1)}))\\
   & = \lim_{k \rightarrow \infty}\frac{r_k^{(l+1)}}{s_k^{(l+1)}} 
   = \lim_{k \rightarrow \infty}\frac{u_k(s_k^{(l+1)})}{u_k(r_k^{(l+1)})} 
   = \lim_{k \rightarrow \infty}s_k^{(l+1)}=0,
 \end{split}
\end{equation}
completing the induction step. In view of \eqref{5.1} and Lemma \ref{lemma5.2}
the iteration must terminate 
after finitely many steps $1 \le l \le l_*$, after which  
\begin{equation*}
    N_k(s_k^{(l_*)},t_k)  \rightarrow 0 \hbox{ as } k \rightarrow \infty
\end{equation*} 
for any sequence $t_k \rightarrow 0$ as $k \rightarrow \infty$
This concludes the proof of Theorem \ref{thm5.1} in the radial case.

\subsection{Proof of Proposition \ref{prop5.1}}
Throughout this section we let $r_k = r_k^{(l+1)}$, etc.,
and we set $r_k^- = r_k^{(l)}$, $s_k^- = s_k^{(l)}$.
Again denote as $v_k(x) = u_k(r_kx)$, $\dot{v}_k(x) = \dot{u}_k(r_kx)$ the scaled 
functions $u_k$, $\dot{u}_k$, respectively. 
As usual we write $v_k(x) = v_k(r)$ for $r = |x|$. Recall that \eqref{5.17}
implies that $v_k(1) = u_k(r_k) \rightarrow \infty$.

\begin{lemma}\label{lemma5.5}
As $k \rightarrow \infty$ we have $v_k(x) - v_k(1) \rightarrow 0$ locally 
uniformly on ${\mathbb R}^2 \setminus \{0\}$.
\end{lemma}

\begin{proof} The function $\tilde{v}_k(x) = v_k(x) - v_k(1)$ satisfies the equation 
\begin{equation*}
    -\Delta \tilde{v}_k = g_k - l_k,
\end{equation*}
where $g_k = \lambda_k r_k^2 v_k e^{v_k^2} $ and with $l_k = r_k^2 \dot{v}_k e^{v_k^2}$.

We claim that $g_k \rightarrow 0$ locally uniformly away from $0$.
Indeed, since $r_k \rightarrow 0$, for any $x$ where $g_k(x) \ge r_k$ we have 
$v_k(x)=u_k(r_kx) \ge \gamma_k$ with constants 
$\gamma_k \rightarrow \infty$ independent of $x$. Hence for any $L > 0$ and any 
$1/L \le |x| \le L$ we either can bound $g_k(x) \le r_k\rightarrow 0$, or
\begin{equation*}
  \begin{split}
       g_k(x) & = \lambda_k r_k^2 v_k(x) e^{v_k^2(x)} 
        = \lambda_k r_k^2 u_k(r_kx) e^{u_k^2(r_kx)} \\
        & = (2\pi)^{-1}|x|^{-2} P_k(r_k|x|)/u_k(r_kx) \le CL^2\gamma_k^{-1}
        \rightarrow 0
  \end{split}
\end{equation*}
as $k \rightarrow \infty$. 
Moreover, \eqref{4.6} and Lemma \ref{lemma4.1} imply
\begin{equation}\label{5.25}
  \begin{split}
    & \int_{B_L\setminus B_{1/L}(0)}|l_k|^2\;dx 
    \le \lambda_k r_k^2 \sup_{1/L \le |x| \le L}e^{v_k^2(x)}
    \bigg(\lambda_k^{-1} \int_{B_{Lr_k}(x_k)}\dot{u}_k^2 e^{u_k^2}\; dx\bigg)\\
    &\le (2\pi)^{-1} L^2
    \sup_{1/L \le |x| \le L}\frac{P_k(r_k|x|)}{u_k^2(r_k|x|)}
    \bigg(\lambda_k^{-1} \int_{\Omega}\dot{u}_k^2 e^{u_k^2}\; dx\bigg)
    \rightarrow 0
  \end{split}
\end{equation}
for any fixed $L > 1$ as $k \rightarrow \infty$.

Since from \eqref{1.6} or \eqref{0.1}, respectively, we also have the uniform $L^2$-bound
\begin{equation*}
    ||\nabla \tilde{v}_k||_{L^2} = ||\nabla u_k||_{L^2} \le C,
\end{equation*}
we may extract a subsequence $(u_k)$ such that 
$\tilde{v}_k \rightarrow \tilde{v}$
weakly in $H^1_{loc}({\mathbb R}^2)$, where $\tilde{v}$ is harmonic 
away from the origin. In addition, $\nabla \tilde{v} \in L^2({\mathbb R}^2)$; since 
the point $x = 0$ has vanishing $H^1$-capacity, we then have 
$\Delta \tilde{v} = 0$ in the distribution sense on all of ${\mathbb R}^2$ 
and $\tilde{v}$ is a smooth, everywhere harmonic function. Again invoking the fact that 
$\nabla \tilde{v} \in L^2({\mathbb R}^2)$, and recalling that 
$\tilde{v}(1)=\tilde{v}_k(1)=0$,
then we see that $\tilde{v}$ vanishes identically; that is,
$\tilde{v}_k \rightarrow 0$ weakly in $H^1_{loc}({\mathbb R}^2)$.
 
Recalling that for radially symmetric functions weak $H^1$-convergence implies 
locally uniform convergence away from the origin, we obtain the claim. 
\end{proof}

Now $\eta_k(x) = v_k(1)(v_k(x) - v_k(1))$ satisfies the equation 
\begin{equation}\label{5.27}
    - \Delta \eta_k = \lambda_k r_k^2 v_k(1)v_k e^{v_k^2} 
                          - r_k^2  v_k(1) \dot{v}_k e^{v_k^2} =:I_k + II_k.
\end{equation}
By Lemma \ref{lemma5.5}
for any $L > 1$ we can bound $\sup_{B_{L}\setminus B_{1/L}} v_k(1)/v_k \le 2$
for sufficiently large $k$. Lemma \ref{lemma4.1}, \eqref{5.1}, 
and \eqref{5.11b} then yield
\begin{equation}\label{5.26}
 \begin{split}
   \int_{B_{L}} |II_k|\; dx
    & \le\int_{B_1} |II_k| dx + \int_{B_{L}\setminus B_{1}}|II_k| dx \\
   & \le o(1) + 2\bigg(\lambda_k \int_{B_{Lr_k}} u_k^2 e^{u_k^2}dx\cdot
       \lambda_k^{-1} \int_{B_{Lr_k}}\dot{u}_k^2 e^{u_k^2}\; dx\bigg)^{1/2} 
   \rightarrow 0,
  \end{split}
\end{equation}
with error $o(1)\rightarrow 0$ as $k \rightarrow \infty$ for any fixed $L > 1$.
Upon estimating 
$v_k(1)v_k e^{v_k^2} \le \max\{v_k^2(1) e^{v_k^2(1)},v_k^2 e^{v_k^2}\}$,
for $1/L \le |x| \le L$ by \eqref{4.6} we can bound the remaining term
\begin{equation}\label{5.28}
  \begin{split}
    I_k(x) \le (2\pi)^{-1} \max\{P_k(r_k), |x|^{-2} P_k(r_k|x|)\} \le C(1+L^2)
  \end{split}
\end{equation}
Moreover, letting $\hat{v}_k = v_k/v_k(1)\rightarrow 1$ in 
$B_{L}\setminus B_{1/L}$, we have
\begin{equation}\label{5.29}
  \begin{split}
    I_k = \lambda_k r_k^2 v_k^2(1) e^{v_k^2(1)} \hat{v}_k 
      e^{v_k^2 - v_k^2(1)} = p_k \hat{v}_k e^{\eta_k(1+\hat{v}_k)},
  \end{split}
\end{equation}
where $p_k = (2\pi)^{-1} P_k(r_k) \ge p_0 > 0$ by \eqref{5.17}. 

Finally, similar to \eqref{5.11e} and in view of \eqref{5.9} we find
\begin{equation}\label{5.31}
  \begin{split} 
   \int_{B_{1/L}(0)} & I_k\; dx 
   = \int_{B_{r_k/L}(0)} \lambda_k u_k(r_k) u_k e^{u_k^2} dx\\ 
   & \le N_k(Lr_k^-,r_k/L) + C\Lambda \frac{u_k(s_k^-)}{u_k(Lr_k^-)}
   \rightarrow 0 ,
  \end{split} 
\end{equation}
if we first let $k \rightarrow \infty$ and then pass to the 
limit $L \rightarrow \infty$.

\begin{lemma}\label{lemma5.6}
There exist a subsequence $(u_k)$ such that $\eta_k \rightarrow \eta_{\infty}$
locally uniformly on ${\mathbb R}^2\setminus \{0\}$ as $k \rightarrow \infty$.
\end{lemma}

\begin{proof}
For any $L > 1$ decompose $ \eta_k = h_k + n_k$ on $B_L \setminus B_{1/L}(0)$,
where $\Delta h_k = 0$ in $B_L \setminus B_{1/L}(0)$, and where $n_k = 0$ 
on $\partial (B_L \setminus B_{1/L}(0))$. In view of \eqref{5.26}, \eqref{5.28},
and passing to a subsequence, if necessary, we may assume that 
$n_k \rightarrow n$ as $k \rightarrow \infty$ 
in $W^{1,q}$ on $B_L \setminus B_{1/L}(0)$ for any 
$q < 2$ and therefore also uniformly by radial symmetry.

On the other hand, letting $h_k^+ = \max\{0,h_k\}$,
from \eqref{5.28} - \eqref{5.29} for sufficiently large $k$ we obtain the estimate
\begin{equation*}
  \begin{split} 
   \int_{B_L \setminus B_{1/L}(0)}h_k^+dx & 
   \le \int_{B_L \setminus B_{1/L}(0)}(\eta_k^+ + |n_k|)\;dx \\
   & \le \int_{B_L \setminus B_{1/L}(0)}e^{(1+\hat{v}_k)\eta_k}dx + C(L)
   \le C(L) < \infty\, .
  \end{split} 
\end{equation*}
From the mean value property of harmonic functions and Harnack's inequality 
we conclude that either $h_k \rightarrow h$ 
locally uniformly on $B_L \setminus B_{1/L}(0)$, or
$h_k \rightarrow -\infty$ and hence $\eta_k \rightarrow -\infty$
locally uniformly on $B_L \setminus B_{1/L}(0)$ as $k \rightarrow \infty$.
But the identity $\eta_k(1)=0$ excludes the latter case, 
and the assertion follows.
\end{proof}

Now we can complete the proof of Proposition \ref{prop5.1}. 
Since $\Delta \eta_k$ 
by \eqref{5.27} - \eqref{5.31} is uniformly bounded in $L^1(B_L(0))$, 
the sequence $(\eta_k)$ is bounded in 
$W^{1,q}(B_L(0))$ for any $q < 2$ and any $L > 1$ and we may assume 
that $\eta_k \rightarrow \eta_0$ also weakly locally in $W^{1,q}$ on 
${\mathbb R}^2$ as $k \rightarrow \infty$.

By Lemmas \ref{lemma5.5} and \ref{lemma5.6} we may then pass to the limit 
$k \rightarrow \infty$ in equation \eqref{5.27} to see that $\eta_{\infty}$ solves the
equation 
\begin{equation}\label{5.33}
   -\Delta \eta_{\infty}= p_{\infty} e^{2\eta_{\infty}} \hbox{ on } {\mathbb R}^2\setminus \{0\},
\end{equation}
for some constant $p_{\infty} = \lim_{k \rightarrow \infty}p_k > 0$.
Moreover, by Lemma \ref{lemma5.6}, and \eqref{5.28} we have
\begin{equation*}
   p_{\infty} e^{2\eta_{\infty}} 
   = \lim_{k \rightarrow \infty} p_k\hat{v}_k^2 e^{\eta_k(\hat{v}_k(x) + 1))}
   = \lim_{k \rightarrow \infty} \hat{v}_k I_k
   = \lim_{k \rightarrow \infty} r_k^2 e_k(r_k\cdot)
\end{equation*}
locally uniformly on ${\mathbb R}^2\setminus \{0\}$.
Thus, with a uniform constant $C$ for any $L > 1$ we have 
\begin{equation*}
 \begin{split}
   p_{\infty} \int_{B_L\setminus B_{1/L}(0)}e^{2 \eta_{\infty}} \; dx & \le 
   \liminf_{k \rightarrow \infty}  \int_{B_{Lr_k}\setminus B_{r_k/L}(0)}e_k\; dx
   \le C \Lambda.
  \end{split}
\end{equation*}
Passing to the limit $L \rightarrow \infty$, we see that  
$e^{2 \eta_{\infty}} \in L^1({\mathbb R}^2)$. 
By \eqref{5.26} and \eqref{5.31} we also have
\begin{equation*}
   \limsup_{k \rightarrow \infty}\int_{B_{1/L}(0)} |\Delta \eta_k|\;dx
   \rightarrow 0
\end{equation*}
as $L \rightarrow \infty$. Hence $\eta_{\infty}$
extends to a distribution solution of \eqref{5.33} on all of ${\mathbb R}^2$. 
Our claim then follows from the Chen-Li \cite{Chen-Li} classification 
of all solutions $\eta_{\infty}$ to equation \eqref{5.33} on 
${\mathbb R}^2$ with $e^{2 \eta_{\infty}} \in L^1({\mathbb R}^2)$ 
in view of radial symmetry of $\eta_{\infty}$ together with the fact that 
$\eta_{\infty}(1) = \eta_k(1) = 0$.

\subsection{The general case}
For the proof of Theorem \ref{thm5.1} in the general case fix an index 
$1 \le i \le i_*$ and let $x_k= x^{(i)}_k \rightarrow x^{(i)}$, 
$0 < r_k = r^{(i)}_k \rightarrow 0$ 
be determined as in Theorem \ref{thm4.1} so that 
$u_k(x_k) = \max_{|x-x_k| \le L r_k}u_k(x)$ for any $L > 0$ and 
sufficiently large $k$ and such that
\begin{equation}\label{5.34}
  \eta_k(x) = \eta^{(i)}_k(x) := 
  u_k(x_k)(u_k(x_k + r_k x ) - u_k(x_k)) 
  \rightarrow \log \left(\frac{1}{1+|x|^2}\right)
\end{equation}
as $k \rightarrow \infty$. 
For each $k$ we may shift the origin so that henceforth
we may assume that $x_k = 0$ for all $k$. 
Denote as $\Omega_k = \Omega_k^{(i)}$ the shifted domain $\Omega$.
We also extend $u_k$ by $0$ outside $\Omega_k$ to obtain
$u_k \in H^1({\mathbb R}^2)$, still satisfying 
\eqref{5.1}.

Again we let 
$e_k = \lambda_k u_k^2 e^{u_k^2}$, $f_k = \lambda_k u_k(0)u_k e^{u_k^2}$,
and for $0 < r < R$ we set 
\begin{equation*}
   \Lambda_k(r) = \int_{B_r} e_k \; dx, \;
   \sigma_k(r) = \int_{B_r} f_k \; dx ,\; 
\end{equation*}
satisfying \eqref{5.3}. 

Also introduce the spherical mean
$\bar{u}_k(r) = \intbar_{\partial B_r} u_k do$ of $u_k$ on 
$\partial B_r$, and so on, and set
$\tilde{e}_k = \lambda_k \bar{u}_k^2e^{\bar{u}_k^2}$. 

The spherical mean
$\bar{w}_k$ of the function
\begin{equation*}
      w_k(x) = u_k(0)(u_k(x) - u_k(0)),
\end{equation*}
satisfies the equation
\begin{equation} \label{5.35}
  \begin{split}
   - \Delta \bar{w}_k = \bar{f}_k - \bar{d}_k,
  \end{split}
\end{equation}
where $\bar{f}_k= \lambda_k u_k(0) \overline{u_k e^{2u_k^2}}$ and where
\begin{equation*}
   \bar{d}_k = u_k(0)  \overline{\dot{u}_k e^{u_k^2}} \rightarrow 0 \hbox{ in } L^1(B_{Lr_k})
\end{equation*}
for any $L > 0$ as $k \rightarrow \infty$ similar to \eqref{5.11b}.

Note that by Jensen's inequality we have
\begin{equation} \label{5.36}
   \tilde{e}_k \le \bar{e}_k;
\end{equation}
hence
\begin{equation*}
   \tilde{\Lambda}_k(r):= \int_{B_r} \tilde{e}_k \; dx \le \Lambda_k(r), \;
   \int_{B_r} \bar{f}_k \; dx = \sigma_k(r).
\end{equation*}
Observe that in analogy with \eqref{5.3} Theorem \ref{thm4.1} implies
\begin{equation} \label{5.37}
  \begin{split}
   \lim_{L \rightarrow \infty} \lim_{k \rightarrow \infty} \tilde{\Lambda}_k(Lr_k) 
   = \lim_{L \rightarrow \infty} \lim_{k \rightarrow \infty} \Lambda_k(Lr_k) 
   = \lim_{L \rightarrow \infty} \lim_{k \rightarrow \infty} \sigma_k(Lr_k)
   = \Lambda_1.
  \end{split}
\end{equation}

To proceed, we need the following estimate similar to the 
gradient estimate of Druet \cite{Druet}, Proposition 2. 
For any $k \in {\mathbb N}$, $x \in \Omega$ we let 
\begin{equation*}
   R_k(x) = \inf_{1 \le j \le i_*} |x - x_k^{(j)}|.
\end{equation*}

\begin{proposition}\label{prop5.2}
There exists a uniform constant $C$ such that for all $y \in \Omega$ there
holds
\begin{equation*}
   \sup_{z \in B_{R_k(y)/2}(y)}|u_k(y) - u_k(z)|  u_k(y) \le C,
\end{equation*}
uniformly in $k \in {\mathbb N}$.
\end{proposition}

The proof of Proposition \ref{prop5.2} is given in the next section.

Recalling that $x_k^{(i)} = 0$, we let  
\begin{equation*}
     \rho_k =\rho^{(i)}_k = \frac12 \inf_{j \neq i} |x_k^{(j)}|,
\end{equation*}     
and we set $\rho_k = \operatorname{diam}({\Omega})$ if $\{j;\;j \neq i\}= \emptyset$, that is, if
there is no other concentration point but $x_k^{(i)}$.
We now use Proposition \ref{prop5.2} to deal with concentrations 
around the point $x_k^{(i)}$ at scales which are small with respect 
to $\rho_k $.

Indeed, for $|x| \le \rho_k$ we have $|x| = R_k(x)$; therefore, by 
Proposition \ref{prop5.2} and Lemma \ref{lemma5.8} below
for any $0 <  r  \le \rho_k$ with a uniform 
constant $C$ there holds
\begin{equation}\label{5.38}
   \sup_{r/2 \le |x| \le r}{u_k^2(x)} - \inf_{r/2 \le |x| \le r}u_k^2(x)
   \le C.
\end{equation}
Hence, in particular, there holds   
\begin{equation}\label{5.39}
   \sup_{r/2 \le |x| \le r}{e^{u_k^2(x)}} \le C e^{\bar{u}_k^2(r)}, 
\end{equation}
and we conclude the estimate 
\begin{equation}\label{5.40}
   \frac{1}{C_3} \sup_{r/2 \le |x| \le r}{u_k^2(x) e^{u_k^2(x)}} 
   \le (1+ \bar{u}_k^2(r)) e^{\bar{u}_k^2(r)} 
   \le C_3\inf_{r/2 \le |x| \le r}(1+u_k^2(x)) e^{u_k^2(x)}
\end{equation}
with a uniform constant $C_3$. In the following we proceed as in \cite{Struwe07};
therefore we only sketch the necessary changes we have to perform in the
present case. 

Because of our choice of origin $x_k^{(i)}=0$ there holds $u_k(x)\le u_k(0)$ for all 
$|x|\le Lr_k$, $k \ge k_0(L)$; hence at this scale there also 
holds the inequality $e_k \le f_k$. 

Similar to Lemma \ref{lemma5.1} with the help of \eqref{5.40} we obtain

\begin{lemma}\label{5.10g}
For any $\varepsilon >0$, if there is a minimal number $0<T_k \le \rho_k$ such that 
$\bar{u}_k(T_k) = \varepsilon u_k(0)$, then 
for any constant $b < 2$ and sufficiently large $k$ there holds 
\begin{equation*}
   \bar{w}_k(r) \le b \log\left(\frac{r_k}{r}\right) \hbox{ on } B_{T_k}
\end{equation*}
and we have
\begin{equation*}
   \lim_{k \rightarrow \infty}\tilde{\Lambda}_k(T_k) =
   \lim_{k \rightarrow \infty}\Lambda_k(T_k) = 
   \lim_{k \rightarrow \infty}\sigma_k(T_k) = 4\pi.
\end{equation*}
\end{lemma}

Next we define for $0\le s<t \le \rho_k$
\begin{equation*}
   N_k(s,t)=\int_{B_t\setminus B_s} e_k dx
   = \lambda_k \int_{B_t\setminus B_s} u_k^2 e^{u_k^2} dx,
\end{equation*}
and
\begin{equation*}
   \tilde{N}_k(s,t)=\int_{B_t\setminus B_s} \tilde{e}_k dx
   = 2\pi \lambda_k \int_s^t r \bar{u}_k^2 e^{\bar{u}_k^2} dr\le N_k(s,t),
\end{equation*}
where we used Jensen's inequality for the last estimate. 
Moreover we let
\begin{equation*}
     P_k(t)=t\frac{\partial}{\partial t} N_k(s,t)=t\int_{\partial B_t} e_k do
\end{equation*}
and
\begin{equation*}
     \tilde{P}_k(t)=t\frac{\partial}{\partial t} \tilde{N}_k(s,t)
     =t\int_{\partial B_t} \tilde{e}_k do 
     =2\pi t^2 \lambda_k \bar{u}_k^2 e^{\bar{u}_k^2} \le P_k(t).
\end{equation*}
The estimate \eqref{5.40} implies
\begin{equation}\label{5.40a}
    N_k(s,t)\le C_3\tilde{N}_k(s,t)+o(1)\ \ \text{and}\ \ 
    P_k(t)\le C_3 \tilde{P}_k(t)+o(1),
\end{equation}
with error $o(1)\rightarrow 0$ as $k\rightarrow \infty$, uniformly in 
$s\le t \le \rho_k$.
Moreover, similar to \cite{Struwe07}, estimate (26), by \eqref{5.40} with 
uniform constants $C_4$, $C_5$ we have 
\begin{equation}\label{5.40b}
    P_k(t)\le C_4 N_k(t/2,t) +o(1) \le C_5P_k(t/2)+o(1).
\end{equation}

If for some $\varepsilon>0$ there is no $T_k=T_k(\varepsilon)\le \rho_k$ 
as in Lemma \ref{5.10g} we continue our argument as described in Case 1
after Proposition \ref{prop5.1g}. Otherwise, we proceed by iteration as in the radially
symmetric case. Choose a sequence 
$\varepsilon_k \downarrow 0$ such that with corresponding numbers 
$s_k=T_k(\varepsilon_k)\le \rho_k$ we have $\bar{u}_k(s_k) \rightarrow \infty$ 
as $k\rightarrow \infty$. Then there holds
\begin{equation*}
    \lim_{k\rightarrow \infty} \Lambda_k(s_k)=\Lambda_1=4\pi
\end{equation*}
and
\begin{equation*}
    \lim_{L\rightarrow \infty} \lim_{k\rightarrow \infty}
         (\Lambda_k(s_k)-\Lambda_k(Lr_k))
    =\lim_{k\rightarrow \infty} \frac{\bar{u}_k(s_k)}{\bar{u}_k(r_k)}
    =\lim_{k\rightarrow \infty} \frac{r_k}{s_k}
    =\lim_{k\rightarrow \infty}s_k=0.
\end{equation*}

By a slight abuse of notation we let $r_k=r_k^{(1)}$, $s_k=s_k^{(1)}$.
Suppose that for some $l\ge 0$ we already have determined numbers 
$r_k^{(1)}<s_k^{(1)}<\ldots <s_k^{(l)}\le \rho_k$ 
such that
\begin{equation}\label{5.12a}
    \lim_{L\rightarrow \infty} 
    \lim_{k\rightarrow \infty} \Lambda_k(s_k^{(l)})=\Lambda_1 l=4\pi l
\end{equation}
and
\begin{equation}\label{5.12b}
    \lim_{L\rightarrow \infty} 
          \lim_{k\rightarrow \infty} (\Lambda_k(s_k^{(l)})-\Lambda_k(Lr_k^{(l)}))
    =\lim_{k\rightarrow \infty} \frac{\bar{u}_k(s_k^{(l)})}{\bar{u}_k(r_k^{(l)})}
    =\lim_{k\rightarrow \infty} \frac{r_k^{(l)}}{s_k^{(l)}}
    =\lim_{k\rightarrow \infty}s_k^{(l)}=0.
\end{equation} 

Similar to Lemma \ref{lemma5.2} we now have the following result.

\begin{lemma}\label{lemma5.2g}
i) Suppose that for some $s_k^{(l)}<t_k \le \rho_k$ there holds
\begin{equation*}
   \sup_{s_k^{(l)} <t<t_k} P_k(t) \rightarrow 0 \hbox{ as } k \rightarrow \infty.
\end{equation*}
Then we have
\begin{equation*}
   \lim_{k \rightarrow \infty}N_k(s_k^{(l)} ,t_k) = 0.
\end{equation*}
ii) Conversely, if for some $s_k^{(l)}<t_k$ and a subsequence $(u_k)$ there holds
\begin{equation*}
   \lim_{k \rightarrow \infty}N_k(s_k^{(l)} ,t_k) = \nu_0 > 0,\
   \lim_{k\rightarrow \infty} \frac{t_k}{\rho_k}=0,
\end{equation*}
then either $\nu_0 \ge \pi$, or we have 
\begin{equation*}
   \liminf_{k \rightarrow \infty}P_k(t_k) \ge \nu_0
\end{equation*}
and
\begin{equation*}
   \lim_{L \rightarrow \infty}\liminf_{k \rightarrow \infty}N_k(s_k^{(l)} ,L t_k) 
   \ge \pi, \;
   \lim_{L \rightarrow \infty}\limsup_{k \rightarrow \infty} N_k(s_k^{(l)} ,t_k/L) = 0.
\end{equation*}
\end{lemma}

\begin{proof}
{\it i)} Because of the estimate \eqref{5.40a} it is enough to prove the Lemma with 
$N_k(s,t)$ and $P_k(t)$ replaced by $\tilde{N}_k(s,t)$ and $\tilde{P}_k(t)$. 
For $s=s_k^{(l)} < t$ we integrate by parts as before to obtain 
\begin{equation} \label{5.11gg}
   2 \tilde{N}_k(s,t) \le \tilde{P}_k(t)
   - 4 \pi \int_s^t \lambda_k r^2 \bar{u}_k'(1 + \bar{u}_k^2)\bar{u}_k e^{\bar{u}_k^2}dr.
\end{equation}

As in the proof of Lemma \ref{lemma5.2} equation \eqref {1.1} yields 
the identity 
\begin{equation} \label{5.11ag}
   - 2\pi t \bar{u}_k(t) \bar{u}_k'(t) 
   = \int_{B_t}\lambda_k \bar{u}_k(t)\overline{u_k e^{u_k^2}} dx 
     - \int_{B_t} \bar{u}_k(t)\overline{\dot{u}_k e^{u_k^2}} dx.
\end{equation}
for any $t \le \rho_k$. Arguing as in \eqref{5.11b} we get that
\begin{equation*}
   \int_{B_t} \bar{u}_k(t)\overline{\dot{u}_k e^{u_k^2}} dx\rightarrow 0 
\end{equation*}
as $k \rightarrow \infty$. In view of \eqref {5.11ag} and Jensen's inequality
at any sequence of points $t=t_k$ where $\bar{u}_k'(t) \ge 0$ then there holds 
\begin{equation}\label{5.11cg}
      0\le\int_{B_t}\lambda_k \bar{u}_k(t)\bar{u}_k e^{\bar{u}_k^2}dx 
       \le \int_{B_t}\lambda_k \bar{u}_k(t)\overline{u_k e^{u_k^2}}dx= o(1).
\end{equation}
Conversely, if $\bar{u}_k'(r) \le 0 = \bar{u}_k'(t_0)$ for 
$t_{k0}=t_0 \le r \le t=t_k$, by \eqref {5.11cg} we can estimate
\begin{equation} \label{5.11dg}
  \begin{split} 
   \int_{B_t} \lambda_k \bar{u}_k(t) \bar{u}_k e^{\bar{u}_k^2} dx 
   & \le \int_{B_t\setminus B_{t_0}} \lambda_k \bar{u}_k^2 e^{\bar{u}_k^2} dx
   + \int_{B_{t_0}} \lambda_k \bar{u}_k(t_0) \bar{u}_k e^{\bar{u}_k^2} dx \\
   & = \tilde{N}_k(t_0,t) + o(1).
  \end{split} 
\end{equation}
Combining the above estimates, similar to \eqref{5.11e} 
for $s=s_k^{(l)} \le r \le t=t_k$ we get 
\begin{equation} \label{5.11eg}
  \begin{split} 
   - 2\pi r \bar{u}_k(r) \bar{u}_k'(r) 
   & = \int_{B_r} \lambda_k \bar{u}_k(r) \bar{u}_k e^{\bar{u}_k^2} dx + o(1) \\
   & \le \tilde{N}_k(s,r) 
     + \int_{B_s} \lambda_k \bar{u}_k(s) \bar{u}_k e^{\bar{u}_k^2} dx + o(1)\\
   & \le \tilde{N}_k(s,r) + \tilde{N}_k(Lr_k,s)
     + \frac{\bar{u}_k(s)}{\bar{u}_k(Lr_k)}\Lambda_k(Lr_k)+ o(1)\\
   & = \tilde{N}_k(s,r)+ o(1),
  \end{split}
\end{equation}
where $o(1) \rightarrow 0$ when first $k \rightarrow \infty$ and then 
$L \rightarrow \infty$. As in \eqref{5.11e} the first inequality is clear when 
$\bar{u}_k' \le 0$ in $[s,r]$, and otherwise follows from \eqref{5.11cg}, 
\eqref{5.11dg}. The second inequality is proved similarly.
Thus we conclude the estimate
\begin{equation} \label{5.12g}
  \begin{split} 
   2 \tilde{N}_k(s,t) & \le \tilde{P}_k(t) 
   + 2 \int_s^t \lambda_k r (1 + \bar{u}_k^2)e^{\bar{u}_k^2}\tilde{N}_k(s,r)dr+ o(1)\\
   & \le \tilde{P}_k(t) + \pi^{-1}\tilde{N}_k(s,t)^2 + o(1).
     \end{split} 
\end{equation}

If we now assume that
\begin{equation*}
   \sup_{s<t<t_k}\tilde{P}_k(t)\le C_3 \sup_{s<t<t_k} P_k(t)+o(1)
   \rightarrow 0 \hbox{ as } k \rightarrow \infty,
\end{equation*}
as in Lemma \ref{lemma5.2} we find the desired decay
\begin{equation*}
   \lim_{k \rightarrow \infty}\tilde{N}_k(s_k^{(l)} ,t_k) = 0
\end{equation*}
when we let $t$ increase from $t=s=s_k^{(l)}$ to $t_k$.

ii) In view of \eqref{5.12g} the second assertion can be proved 
as in Lemma \ref{lemma5.2}. 
\end{proof}

By the preceding result it now suffices to consider the following two cases.
In {\bf Case A} for any sequence $t_k = o(\rho_k)$ we have
\begin{equation*}
   \sup_{s_k^{(l)} <t< t_k} P_k(t) \rightarrow 0 \hbox{ as } k \rightarrow \infty,
\end{equation*}
and then in view of Lemma \ref{lemma5.2g} also 
\begin{equation}\label{5.12k}
  \lim_{L\rightarrow \infty}\lim_{k\rightarrow \infty} N_k(s_k^{(l)},\rho_k/L) = 0,
\end{equation}
thus completing the concentration analysis at scales up to $o(\rho_k)$.

In {\bf Case B} for some $s_k^{(l)}<t_k\le \rho_k$ there holds
\begin{equation*}
  \limsup_{k\rightarrow \infty} N_k(s_k^{(l)},t_k)>0,\ \
  \lim_{k\rightarrow \infty} \frac{t_k}{\rho_k}=0.
\end{equation*}
Then, as in the radial case, from Lemma \ref{lemma5.2g} we infer that for a 
subsequence $(u_k)$ and suitable numbers $r_k^{(l+1)} \in ]s_k^{(l)},t_k[$ we have
\begin{equation}\label{5.12h}
  \lim_{L\rightarrow \infty}\lim_{k\rightarrow \infty} N_k(s_k^{(l)},Lr_k^{(l+1)})
  \ge \pi, \ \ \liminf_{k\rightarrow \infty} P_k(r_k^{(l+1)})>0;
\end{equation}
in particular, $\bar{u}_k(r_k^{(l+1)}) \rightarrow \infty$ as $k\rightarrow \infty$. 
Moreover, as in Lemma \ref{lemma5.2g} the bound \eqref{5.12h} implies that 
$r_k^{(l+1)}/s_k^{(l)} \rightarrow \infty$ as $k\rightarrow \infty$.
Indeed, assume by contradiction that $r_k^{(l+1)}\le Ls_k^{(l)}$ 
for some $L>0$. Then from \eqref{5.40}, \eqref{5.40b}, and recalling that  
$N_k(s_k^{(l)}/2,s_k^{(l)}) \rightarrow 0$ as $k\rightarrow \infty$
we obtain that $P_k(r_k^{(l+1)}) \rightarrow 0$ 
contrary to \eqref{5.12h}. Also note that
\begin{equation}\label{5.12i}
   \lim_{L\rightarrow \infty} \limsup_{k\rightarrow \infty} N_k(s_k^{(l)}, r_k^{(l+1)}/L)
   =\lim_{k\rightarrow \infty} \frac{r_k^{(l+1)}}{\rho_k}
   =\lim_{k\rightarrow \infty} \frac{t_k}{\rho_k}=0.
\end{equation}

Moreover, we have the following analogue of Proposition \ref{prop5.1}.

\begin{proposition} \label{prop5.1g} 
There exist a subsequence $(u_k)$ such that
\begin{equation*}
  \eta^{(l+1)}_k(x) := 
  \bar{u}_k(r_k^{(l+1)})(u_k(r_k^{(l+1)} x ) - \bar{u}_k(r_k^{(l+1)})) 
  \rightarrow \eta(x)
\end{equation*}
locally uniformly on ${\mathbb R}^2\setminus \{0\}$ as $k \rightarrow \infty$, 
where $\eta$ solves \eqref{4.4}, \eqref{4.4a}.
\end{proposition}

Proposition \ref{prop5.1g} is a special case of Proposition \ref{blow} below, whose proof 
will be presented in Section \ref{proofblow}.

From Proposition \ref{prop5.1g} the desired energy quantization result at the scale 
$r_k^{(l+1)}$ follows as in the radial case. 
If $\rho_k \ge \rho_0 > 0$ we can argue as in \cite{Struwe07}, p.~416,
to obtain numbers $s_k^{(l+1)}$ satisfying \eqref{5.12a}, \eqref{5.12b} for $l+1$
and such that $\bar{u}_k(s_k^{(l+1)})\rightarrow \infty$ as $k \rightarrow \infty$.
By iteration we then establish \eqref{5.12a}, \eqref{5.12b} up to $l = l_0$ for some 
maximal index $l_0 \ge 0$ and thus complete the concentration analysis 
near the point $x^{(i)}$.

If $\rho_k\rightarrow 0$ as $k \rightarrow \infty$, we distinguish the following two cases.
In {\bf Case 1} for some $\varepsilon_0>0$ and all $t \in [r_k^{(l+1)},\rho_k]$ there
holds $\bar{u}_k(t)\ge\varepsilon_0\bar{u}_k(r_k^{(l+1)})$. The decay estimate
that we established in Lemma \ref{5.10g} then remains valid throughout this range
and \eqref{5.12a} holds true for any choice $s_k^{(l+1)}=o(\rho_k)$ for $l=l+1$.
Again the concentration analysis at scales up to $o(\rho_k)$ is complete.
In {\bf Case 2}, for any $\varepsilon>0$ there is a minimal 
$T_k=T_k(\varepsilon)\in [r_k^{(l+1)},\rho_k]$ as in Lemma \ref{5.10g} such that
$\bar{u}_k(T_k)=\varepsilon\bar{u}_k(r_k^{(l+1)})$. 
Then as before we can define numbers $s_k^{(l+1)} < \rho_k$ with 
$\bar{u}_k(s_k^{(l+1)})\rightarrow \infty$ as $k \rightarrow \infty$
so that \eqref{5.12a}, \eqref{5.12b} also hold true for $l+1$, and we proceed
by iteration up to some maximal index $l_0 \ge 0$ where either Case 1 
or Case A holds with final radius $r^{(l_0)}$.

For the concentration analysis at the scale $\rho_k$
first assume that for some number $L\ge 1$ there is a sequence $(x_k)$ such that 
$\rho_k/L \le R_k(x_k)\le |x_k|\le L\rho_k$ and
\begin{equation}\label{sub}
   \lambda_k |x_k|^2 u_k^2(x_k)e^{u_k^2(x_k)}\ge \nu_0>0.
\end{equation}
By Proposition \ref{prop5.2} we may assume that $|x_k|=\rho_k$.
As in \cite{Struwe07}, Lemma 4.6, we then have  
$\bar{u}_k(\rho_k)/\bar{u}_k(r_k^{(l_0)})\rightarrow 0$ as $k\rightarrow \infty$, 
ruling out Case 1; that is, at scales up to $o(\rho_k)$ we end with Case A.
The desired quantization result at the scale $\rho_k$ then is a consequence of the 
following result that we demonstrate in Section \ref{proofblow} below.

\begin{proposition} \label{blow} 
Assuming \eqref{sub}, there exists a finite set $S_0\subset \mathbb{R}^2$ and a subsequence 
$(u_k)$ such that
\begin{equation*}
  \eta_k(x) := 
  u_k(x_k)(u_k(\rho_k x ) - u_k(x_k)) \rightarrow \eta(x)
\end{equation*}
locally uniformly on ${\mathbb R}^2\setminus S_0$ as $k \rightarrow \infty$, where 
$\eta$ solves \eqref{4.4}, \eqref{4.4a}.
\end{proposition}

By Proposition \ref{blow} in case of \eqref{sub} there holds
\begin{equation*}
     \lim_{L\rightarrow \infty} \limsup_{k\rightarrow \infty} 
     \int_{\{x\in \Omega; \frac{\rho_k}{L}\le R_k(x) \le |x| \le L\rho_k\}} e_k dx
     =\Lambda_1=4\pi.
\end{equation*}
Letting 
\begin{equation*}
    X_{k,1}=X_{k,1}^{(i)}=\{ x_k^{(j)}; \exists C>0: |x_k^{(j)}|\le C\rho_k\ 
    \ \text{for all} \ \ k\}
\end{equation*}
and carrying out the above blow-up analysis up to scales of order $o(\rho_k)$ 
also on all balls of center $x_k^{(j)}\in X_{k,1}$, then from \eqref{5.12b} we have 
\begin{equation*}
    \lim_{L\rightarrow \infty}\lim_{k\rightarrow \infty} 
    \Lambda_k (L\rho_k)= \Lambda_1(1+I_1)=4\pi (1+I_1),
\end{equation*}
where $I_1$ is the total number of bubbles concentrating at the points 
$x_k^{(j)}\in X_{k,1}^{(i)}$ at scales $o(\rho_k)$.

On the other hand, if \eqref{sub} fails to hold clearly we have
\begin{equation}\label{sub1}
     \lim_{L\rightarrow \infty} \limsup_{k\rightarrow \infty} 
     \int_{\{x\in \Omega; \frac{\rho_k}{L}\le R_k(x) \le |x| \le L\rho_k\}} e_k dx=0,
\end{equation}
and the energy estimate at the scale $\rho_k$ again is complete.

In order to deal with secondary concentrations around $x_k^{(i)}=0$ at scales 
exceeding $\rho_k$, with $X_{k,1}$ defined as above we let
\begin{equation*}
    \rho_{k,1}=\rho_{k,1}^{(i)}=\frac{1}{2}\inf_{\{j; x_k^{(j)} 
    \notin X_{k,1}\}}|x_k^{(j)}|;
\end{equation*}
again we set $\rho_{k,1}= \operatorname{diam}({\Omega})$, 
if $\{j; x_k^{(j)} \notin X_{k,1}\}=\emptyset$. From this definition 
it follows that $\rho_{k,1}/\rho_k\rightarrow \infty$ as $k\rightarrow \infty$. 
Then either we have 
\begin{equation*}
    \lim_{L\rightarrow \infty} \limsup_{k\rightarrow \infty} 
    N_k(L\rho_k,\frac{\rho_{k,1}}{L})=0,
\end{equation*}
and we iterate to the next scale; or there exist radii $t_k \le \rho_{k,1}$ such that 
$t_k/\rho_k\rightarrow \infty$, $t_k/\rho_{k,1}\rightarrow 0$ 
as $k\rightarrow \infty$ and a subsequence $(u_k)$ such that
\begin{equation}\label{sub1a}
    P_k(t_k)\ge \nu_0>0\ \text{ for all } k.
\end{equation}
The argument then depends on whether \eqref{sub} or \eqref{sub1} holds.
In case of \eqref{sub}, as in \cite{Struwe07}, Lemma 4.6, the bound \eqref{sub1a} and
Proposition \ref{blow} imply that $\bar{u}_k(t_k)/\bar{u}_k(\rho_k)\rightarrow 0$ 
as $k\rightarrow 0$. Then all the previous results remain true for 
$r\in [L\rho_k,\rho_{k,1}]$ for sufficiently large $L$, and we can continue as before 
to resolve concentrations in this range of scales.

In case of \eqref{sub1} we further need to distinguish whether Case A or Case 1 holds
at the final stage of our analysis at scales $o(\rho_k)$. 
In fact, for the following estimates we also consider all 
points $x_k^{(j)}\in X_{k,1}^{(i)}$ in place of $x_k^{(i)}$.
Recalling that in Case A we have \eqref{5.12b} and \eqref{5.12k}, and arguing as above
in Case 1, on account of \eqref{sub1} 
for a suitable sequence of numbers $s_{k,1}^{(0)}$ such that 
$s_{k,1}^{(0)}/\rho_k\rightarrow \infty$, $t_k/s_{k,1}^{(0)}\rightarrow \infty$ as 
$k\rightarrow \infty$ we find 
\begin{equation*}
   \lim_{L\rightarrow \infty} \lim_{k\rightarrow \infty} 
   \Big(\Lambda(s_{k,1}^{(0)})-\sum_{x_k^{(j)}\in X_{k,1}^{(i)}} 
   \Lambda_k^{(j)}(Lr_k^{(l_0^{(j)})})\Big)=0,
\end{equation*}
where $\Lambda_k^{(j)}(r)$ and $r_k^{(l_0^{(j)})}$ are computed as above with respect 
to the concentration point $x_k^{(j)}$. 
In particular, with such a choice of $s_{k,1}^{(0)}$ we find the 
intermediate quantization result 
\begin{equation*}
    \lim_{k\rightarrow \infty} \Lambda_k (s_{k,1}^{(0)})=\Lambda_1 I_1=4\pi I_1
\end{equation*}
analogous to \eqref{5.12a}, where $I_1$ is defined as above.
Moreover, in Case 1 we can argue as in \cite{Struwe07}, 
Lemma 4.8, to conclude that $\bar{u}_k(t_k)/\bar{u}_k(r_k^{(l_0^{(j)})})\rightarrow 0$ 
as $k\rightarrow 0$; therefore, similar to \eqref{5.12b} in Case A, we can achieve that  
\begin{equation*}
   \lim_{k\rightarrow \infty} \frac{\bar{u}_k(s_{k,1}^{(0)})}
        {\bar{u}_k(r_k^{(l_0^{(j)})})}
   =\lim_{k\rightarrow \infty} \frac{r_k^{(l_0^{(j)})}}{s_{k,1}^{(0)}}=0
\end{equation*}
for all $x_k^{(j)}\in X_{k,1}^{(i)}$ where Case 1 holds.

We then finish the argument by iteration. For $l\ge 2$ we inductively define the sets
\begin{equation*}
   X_{k,l}=X_{k,l}^{(i)}=\{ x_k^{(j)}; \exists C>0: |x_k^{(j)}|\le C\rho_{k,l-1}\ 
   \text{ for all } k\}
\end{equation*}
and we let
\begin{equation*}
    \rho_{k,l}=\rho_{k,l}^{(i)}=\frac{1}{2}\inf_{\{j; x_k^{(j)}\notin X_{k,l}^{(i)}\}}
    |x_k^{(j)}|;
\end{equation*}
as before, we set $\rho_{k,l}= \operatorname{diam}({\Omega})$, if 
$\{j; x_k^{(j)}\notin X_{k,l}^{(i)}\}=\emptyset$. Iteratively performing
the above analysis 
at all scales $\rho_{k,l}$, thereby exhausting all concentration points $x_k^{(j)}$, 
upon passing to further subsequences, we finish the proof of Theorem \ref{thm1.1}.

\subsection{Proof of Proposition \ref{prop5.2}}
We argue as in \cite{Struwe07}, thereby closely following the proof of 
Druet \cite{Druet}, Proposition 2.
Suppose by contradiction that
\begin{equation}\label{5.41}
  L_k:= \sup_{y\in \Omega} 
        \bigg(\sup_{z \in B_{R_k(y)/2}(y)}|u_k(y) - u_k(z)|u_k(y)\bigg)
  \rightarrow \infty \hbox{ as } k \rightarrow \infty\; .
\end{equation}
Let $y_k \in \Omega$, $z_k \in B_{R_k(y_k)/2}(y_k)$ satisfy
\begin{equation}\label{5.42}
  |u_k(y_k) - u_k(z_k)| u_k(y_k) 
  \ge L_k/2.
\end{equation}

\begin{lemma}\label{lemma5.7}
We have $u_k(y_k) \rightarrow \infty$ as $k \rightarrow \infty$. 
\end{lemma}

\begin{proof}
Suppose by contradiction that $u_k(y_k) \le C < \infty$.
From \eqref{5.42} we then find that $u_k(z_k) \rightarrow \infty$ 
as $k \rightarrow \infty$. 
Also letting $\hat{z}_k = (y_k + z_k)/2$, we now observe that 
\begin{equation*}
R_k(z_k), R_k(\hat{z}_k) \ge R_k(y_k)/2 > |y_k-z_k| 
=2|y_k-\hat{z}_k|=2|\hat{z}_k-z_k|;
\end{equation*}
hence
\begin{equation*}
   y_k \in B_{R_k(\hat{z}_k)/2}(\hat{z}_k)\; ,
   \hat{z}_k \in B_{R_k(z_k)/2}(z_k).
\end{equation*}
But then the estimate
\begin{equation*}
    \frac{L_k}{2u_k(y_k)} 
    \le |u_k(z_k) - u_k(y_k)| \le |u_k(z_k)-u_k(\hat{z}_k)|+|u_k(\hat{z}_k)-u_k(y_k)|,
\end{equation*}
our assumption that $u_k(y_k) \le C$, and our choice of $y_k$, $z_k$ imply
\begin{equation*}
 \frac{1}{L_k}\big(|u_k(\hat{z}_k) - u_k(y_k)|u_k(\hat{z}_k)
                  +|u_k(\hat{z}_k) - u_k(z_k)|u_k(z_k)\big) 
 \rightarrow \infty
\end{equation*}
as $k \rightarrow \infty$, and a contradiction to \eqref{5.41} results.
\end{proof}

A similar reasoning also yields the following result.

\begin{lemma}\label{lemma5.8}
There exists an absolute constant $C$ such that
\begin{equation*}
 \sup_{z \in B_{R_k(y)/2}(y)}|u_k^2(y) - u_k^2(z)| \le CL_k,
\end{equation*}
uniformly in $y \in \Omega$. In fact, we may take $C=6$.
\end{lemma}

\begin{proof}
From the identity 
\begin{equation*}
  \begin{split}
   u_k^2(y) - u_k^2(z) & = (u_k(y) - u_k(z))(u_k(y)+u_k(z)) \\
   & = 2(u_k(y) - u_k(z))u_k(y)-(u_k(y) - u_k(z))^2 
  \end{split}
\end{equation*}
we conclude the bound
\begin{equation*}
 |u_k^2(y) - u_k^2(z)|\le 2L_k+(u_k(y) - u_k(z))^2
\end{equation*}
for all $y \in \Omega$, $z\in B_{R_k(y)/2}(y)$, and we are done 
unless for some such points $y$ and $z$ there holds
$(u_k(y) - u_k(z))^2 \ge 4L_k$.
Suppose we are in this case. From \eqref{5.41} we then obtain the estimate 
$u_k(y) \le \sqrt{L_k}/2$ and hence $u_k(z) \ge 2\sqrt{L_k}$.
Letting $\hat{z} = (y + z)/2$, as in the proof of Lemma \ref{lemma5.7}
above we observe that $R_k(z), R_k(\hat{z}) \ge R_k(y)/2 \ge |y-z|$ and
\begin{equation*}
   y \in B_{R_k(\hat{z})/2}(\hat{z})\; ,
   \hat{z} \in B_{R_k(z)/2}(z).
\end{equation*}
Since $u_k(z) \ge 2\sqrt{L_k}$, the bound \eqref{5.41} implies 
that $u_k(\hat{z}) \ge 3\sqrt{L_k}/2$. But then, upon estimating
\begin{equation*}
  \begin{split}
 2 L_k 
 & \ge |u_k(y) - u_k(\hat{z})|u_k(\hat{z})+|u_k(\hat{z}) - u_k(z)|u_k(z)\\ 
 & \ge 3 \sqrt{L_k} |u_k(y) - u_k(z)|/2 \ge 3 L_k
  \end{split}
\end{equation*}
we arrive at the desired contradiction.
\end{proof}

From Theorem \ref{thm4.1} and Lemma \ref{lemma5.7} it follows that 
$s_k:=R_k(y_k)\rightarrow 0$ as $k \rightarrow \infty$. Set 
\begin{equation*}
   \Omega_k = \{y; y_k + s_k y \in \Omega\}
\end{equation*}
and scale 
\begin{equation*}
   v_k(y) = u_k(y_k + s_k y),\; \dot{v}_k(y) = \dot{u}_k(y_k + s_k y),\; y \in \Omega_k.
\end{equation*}
Letting $x_k^{(i)}$ be as in the statement of Theorem \ref{thm4.1}, we set
\begin{equation*}
   y_k^{(i)} = \frac{x_k^{(i)} -y_k}{s_k}, 1 \le i \le i_*,
\end{equation*}
and let 
\begin{equation*}
   S_k = \{y_k^{(i)}; 1 \le i \le i_*\}. 
\end{equation*}
Note that in the scaled coordinates we have 
\begin{equation*}
   dist(0, S_k) = \inf\{|y_k^{(i)}|;1 \le i \le i_*\}= 1.
\end{equation*}
Also let 
\begin{equation*}
   p_k = \frac{z_k -y_k}{s_k} \in B_{1/2}(0).
\end{equation*}
Then there holds
\begin{equation}\label{5.43}
  \begin{split}
   L_k/2 & \le |v_k(p_k) - v_k(0)|v_k(0) \\
   & \le \sup_{y\in \Omega_k} 
   \bigg(\sup_{z \in B_{dist(y,S_k)/2}(y)}|v_k(y)-v_k(z)|v_k(y)\bigg)
   = L_k\; ;
  \end{split}
\end{equation}
moreover, from Lemma \ref{lemma5.8} we have
\begin{equation}\label{5.44}
 \sup_{y\in \Omega_k} 
 \bigg(\sup_{z \in B_{dist(y,S_k)/2}(y)}|v_k^2(y) - v_k^2(z)|\bigg) \le CL_k.
\end{equation}

Since $s_k=R_k(y_k)\rightarrow 0$ we may assume that as $k \rightarrow \infty$
the domains $\Omega_k$ exhaust the domain
\begin{equation*}
       \Omega_0 = {\mathbb R} \times ]-\infty,R_0[,
\end{equation*}
where $0 < R_0 \le \infty$. We also may assume that as $k \rightarrow \infty$
either $|y_k^{(i)}|\rightarrow \infty$ or $y_k^{(i)} \rightarrow y^{(i)}$, 
$1 \le i \le i_*$, and we let $S_0$ be the set of accumulation points of $S_k$,
satisfying $dist(0, S_0) = 1$. 
For $R > 0$ denote as 
\begin{equation*}
   K_R = K_{k,R} 
       =\Omega_k \cap B_R(0) \setminus \bigcup_{y \in S_k} B_{1/R}(y).
\end{equation*}
 
Note that we have
\begin{equation*}
  R_k(y_k+s_ky) = s_k\; dist(y,S_k) \ge s_k/R \ \hbox{ for all } y \in K_R.
\end{equation*}
Thus \eqref{4.6} in Theorem \ref{thm4.1} implies the bound
\begin{equation}\label{5.45}
   \lambda_k s_k^2 v_k^2(y)e^{v_k^2(y)} \le C = C(R) \ \hbox{ for all } y \in K_R.
\end{equation}
Finally, letting 
\begin{equation}\label{5.46}
   - v_k\Delta v_k 
   = \lambda_k s_k^2 v_k^2 e^{v_k^2} - s_k^2 \dot{v}_k v_k e^{v_k^2} 
   =: I_k + II_k,
\end{equation}
by \eqref{5.1} we can estimate 
\begin{equation}\label{5.46a}
   ||I_k||_{L^1(\Omega_k)} 
    = \lambda_k\int_{\Omega_k}s_k^2 v_k^2 e^{v_k^2}dy
    = \lambda_k\int_{\Omega}u_k^2 e^{u_k^2}dx \le C;
\end{equation}
moreover, by H\"older's inequality and Lemma \ref{lemma4.1} we have
\begin{equation}\label{5.46b}
   ||II_k||^2_{L^1(\Omega_k)} 
   \le \bigg(\lambda_k\int_{\Omega}u_k^2 e^{u_k^2}dx \bigg)\cdot
   \bigg(\lambda_k^{-1}\int_{\Omega}u_t^2 e^{u_k^2}dx\bigg)
   \rightarrow 0
\end{equation}
as $k \rightarrow \infty$. 
In view of \eqref{5.45} we also have the local $L^2$-bounds 
\begin{equation}\label{5.46c}
  \begin{split}
   ||I_k||^2_{L^2(K_R)} 
   & \le C \sup_{K_R}\big(\lambda_k s_k^2 v_k^2 e^{v_k^2}\big)
   \cdot \bigg(\lambda_k\int_{\Omega_k}s_k^2 v_k^2 e^{v_k^2}dy\bigg)\\
   & \le C(R)\lambda_k\int_{\Omega}u_k^2 e^{u_k^2}dx \le C(R),
  \end{split}
\end{equation}
while Lemma \ref{lemma4.1} implies 
\begin{equation}\label{5.46d}
   ||II_k||^2_{L^2(K_R)} 
   \le C \sup_{K_R}\big(\lambda_k s_k^2 v_k^2 e^{v_k^2}\big)\cdot
   \bigg(\lambda_k^{-1}\int_{\Omega}\dot{u}_k^2 e^{u_k^2}dx\bigg)
   \rightarrow 0 
\end{equation}
as $k \rightarrow \infty$, for any $R > 0$.
Similarly, for any $R > 0$ we find  
\begin{equation}\label{5.47}
    ||\Delta v_k||_{L^2(K_R)} \rightarrow 0 \ (k \rightarrow \infty).
\end{equation}
Also observe that \eqref{5.1} yields the uniform bound
\begin{equation}\label{5.48}
    ||\nabla v_k||_{L^2(\Omega_k)} \le C.
\end{equation}
 
\begin{lemma}\label{lemma5.9}
We have $R_0 = \infty$; that is, $\Omega_0 = {\mathbb R}^2$.
\end{lemma}

\begin{proof}
Suppose by contradiction that $R_0 < \infty$. Choosing $R = 2 R_0$,
from \eqref{1.2} and \eqref{5.44} we conclude the uniform bound 
\begin{equation*}
   \sup_{y \in K_R}v_k^2(y) \le C L_k
\end{equation*}
with $C = C(R)$. Letting $w_k = \frac{v_k}{\sqrt{L_k}}$, 
we then have $0 \le w_k \le C$,
while \eqref{5.47} and \eqref{5.48} give
\begin{equation*}
   ||\nabla w_k||_{L^2(\Omega_k)} + ||\Delta w_k||_{L^2(K_R)} 
   \rightarrow 0 \hbox{ as } k \rightarrow \infty.
\end{equation*}
Since $w_k = 0$ on $\partial \Omega_k \cap K_R$, it follows that 
$w_k \rightarrow 0$ locally uniformly, contradicting the fact that 
$|w_k(p_k) - w_k(0)|w_k(0) \ge 1/2$.
\end{proof}

\begin{lemma}\label{lemma5.10}
As $k \rightarrow \infty$ we have
\begin{equation*}
   \frac{v_k}{v_k(0)} \rightarrow 1 \hbox{ locally uniformly in } 
   {\mathbb R}^2 \setminus S_0.
\end{equation*}
\end{lemma}

\begin{proof}
Recall from Lemma \ref{lemma5.7} that
\begin{equation*}
   c_k := u_k(y_k) = v_k(0) \rightarrow \infty \hbox{ as }
   k \rightarrow \infty.
\end{equation*}
Letting $w_k = c_k^{-1}v_k$, from \eqref{5.47} and \eqref{5.48} 
for any $R > 0$ then we have
\begin{equation*}
   ||\nabla w_k||_{L^2(\Omega_k)} + ||\Delta w_k||_{L^2(K_R)} 
   \rightarrow 0 \hbox{ as } k \rightarrow \infty,
\end{equation*}
and we conclude that $w_k$ converges locally uniformly on 
${\mathbb R}^2 \setminus S_0$ to a constant limit function $w$. 
Recalling that $dist(0, S_0) = 1$, we obtain that $w \equiv w(0) = 1$,
as claimed. 
\end{proof}

Define
\begin{equation*}
   \tilde{v}_k(y) = \frac{1}{L_k}(v_k(y) - v_k(0)) v_k(0).
\end{equation*}
We claim that $\tilde{v}_k$ grows at most logarithmically. To see this, 
let $s_0 \ge 2\sup_i |y^{(i)}|$ and fix $q = 3/2$. For any fixed $R > 0$, any
$y \in K_R$ with $|y| \ge q^{L} s_0$ let $y_l = q^{l-L}y$, $0 \le l \le L$,
so that $y_{l-1} \in B_{dist(y_l,S_k)/2}(y_l)$ for all $l \ge 1$ 
and sufficiently large $k$. Note that we have 
$|v_k(y_0) - v_k(0)| v_k(0) \le CL_k$.
By Lemma \ref{lemma5.10} with error $o(1) \rightarrow 0$ as 
$k \rightarrow \infty$ then we can estimate 
\begin{equation}\label{5.49}
  \begin{split}
   |\tilde{v}_k(y)| & \le \frac{1}{L_k} 
   \sum_{l=1}^L|v_k(y_l) - v_k(y_{l-1})| v_k(0) + C\\
   & \le \frac{1 + o(1)}{L_k} \sum_{l=1}^L|v_k(y_l) - v_k(y_{l-1})| v_k(y_l)
   + C\\
   & \le C + (1 + o(1))L  \le C + (C + o(1))\log |y|.
  \end{split}
\end{equation}
Moreover, from \eqref{5.46c}, \eqref{5.46d} and Lemma \ref{lemma5.10} 
for any $R > 0$ with a constant $C = C(R)$ we obtain 
\begin{equation}\label{5.50}
   ||\Delta \tilde{v}_k||_{L^2(K_R)}  \leq
   \sup_{K_R}\bigg(\frac{v_k(0)}{L_k v_k}\bigg)||v_k \Delta v_k||_{L^2(K_R)}
   \le C \sup_{K_R}\bigg(\frac{v_k(0)}{L_k v_k}\bigg) \rightarrow 0
\end{equation}
as $k \rightarrow \infty$. Thus we may assume that 
$\tilde{v}_k \rightarrow \tilde{v}$ locally uniformly away from $S_0$,
where $\tilde{v}$ satisfies
\begin{equation}\label{5.51}
   \Delta \tilde{v} = 0,\; \tilde{v}(0) = 0,\; 
   \sup_{B_{1/2}(0)} \tilde{v} \ge 1/2,\;
   |\tilde{v}(y)| \le C + C\log(1+|y|).
\end{equation}

Fix any point $x_0 \in S_0$. For any $r > 0$ upon estimating 
$v_k(0)v_k e^{v_k^2} \le \max\{v_k^2(0) e^{v_k^2(0)},v_k^2 e^{v_k^2}\}$
we have 
\begin{equation*}
    L_k\int_{B_r(x_0)} |\Delta \tilde{v}_k|\;dx 
    = \int_{B_r(x_0)} v_k(0)|\Delta v_k|\;dx = I_k + II_k,
\end{equation*}
where
\begin{equation*}
  \begin{split}
    I_k = \int_{B_r(x_0)}\lambda_k s_k^2 v_k(0)v_k e^{v_k^2}dx
      & \le C \lambda_k s_k^2 v_k^2(0) e^{v_k^2(0)} 
        + \lambda_k \int_{B_r(x_0)} s_k^2 v_k^2 e^{v_k^2}\;dx\\
      & \le C \lambda_k R_k^2(y_k) u_k^2(y_k)e^{u_k^2(y_k)}
        + \lambda_k \int_{\Omega} u_k^2 e^{u_k^2}\;dx \le C
  \end{split}
\end{equation*}
by Theorem \ref{thm4.1} and \eqref{5.40}. Similarly, by H\"older's inequality  
\begin{equation*}
    |II_k|^2 = \big|\int_{B_r(x_0)} s_k^2  v_k(0) |\dot{v}_k| e^{v_k^2}dx\big|^2
    \le C \lambda_k^{-1}\int_{B_r(x_0)}s_k^2 \dot{v}_k^2 e^{v_k^2}dx \rightarrow 0.
\end{equation*}
as $k \rightarrow \infty$. 
It follows that
$\Delta \tilde{v}_k \rightarrow 0$ in $L^1_{loc}({\mathbb R}^2)$ 
as $k \rightarrow \infty$.
The sequence $(\tilde{v}_k)$ therefore is uniformly locally bounded in 
$W^{1,q}$ for any $q < 2$ and the limit 
$\tilde{v}\in W^{1,q}_{loc}({\mathbb R}^2)$ extends as a weakly 
harmonic function to all of ${\mathbb R}^2$. 
The mean value property together with the logarithmic growth 
condition \eqref{5.51} then implies that $\tilde{v}$ is a constant; 
see for instance \cite{Adi-Robert-Struwe}, Theorem 2.4.
That is, $\tilde{v} \equiv \tilde{v}(0) = 0$. But by \eqref{5.51} we have
$\sup_{B_{1/2}(0)} |\tilde{v}| \ge 1/2$, which is the desired contradiction
and completes the proof of Proposition \ref{prop5.2}.

\subsection{Proof of Proposition \ref{blow}}\label{proofblow}
We follow closely the proof of Proposition $4.7$ in \cite{Struwe07}. 
Fix an index $i\in \{1,\ldots,i_\star\}$ and write $r_k=\rho_k$. Define
\begin{equation*}
   v_k(y)=u_k(x_k^{(i)}+r_ky)
\end{equation*}
where $y\in \Omega_k=\Omega_k^{(i)}=\{y; x_k^{(i)}+r_ky \in \Omega\}$. Also let
\begin{equation*}
   y_k^{(j)}=\frac{x_k^{(j)}-x_k^{(i)}}{r_k}
\end{equation*}
and
\begin{equation*}
   S_k=S_k^{(i)}=\{ y_k^{(j)};1\le j\le i_\star\}.
\end{equation*}
By choosing a subsequence we may assume that as $k\rightarrow \infty$ either 
$|y_k^{(j)}|\rightarrow \infty$ or $y_k^{(j)}\rightarrow y^{(j)}$, $1\le j\le i_\star$, 
and we let $S_0=S_0^{(i)}$ be the set of accumulation points of $S_k$. 
Note that $0\in S_0$. Finally we let 
\begin{equation*}
y_k^{(0)}=\frac{x_k-x_k^{(i)}}{r_k}
\end{equation*}
be the scaled points $x_k$ for which \eqref{sub} holds and which satisfy $|y_k^{(0)}|=1$. 
Choosing another subsequence we may assume that $y_k^{(0)}\rightarrow y^{(0)}$ 
as $k\rightarrow \infty$. 

Recalling that $v_k(y_k^{(0)})\rightarrow \infty$ by \eqref{sub} and observing that 
$\mathbb{R}^2\backslash S_0$ is connected, from Proposition \ref{prop5.2} and a 
standard covering argument we obtain that 
\begin{equation}\label{conv}
    v_k-v_k(y_k^{(0)})\rightarrow 0\ \ \text{as}\ \ k\rightarrow \infty
\end{equation}
locally uniformly on $\mathbb{R}^2\backslash S_0$. Moreover, as $k\rightarrow \infty$, the 
sets $\Omega_k$ exhaust all of $\mathbb{R}^2$.
  
Next we note that $\eta_k$ satisfies the equation
\begin{equation}\label{eq}
   -\Delta \eta_k=\lambda_k r_k^2 v_k(y_k^{(0)})v_ke^{v_k^2}
   -r_k^2 v_k(y_k^{(0)})\dot{v}_ke^{v_k^2}=I_k+II_k
\end{equation}
on $\Omega_k$. For $L > 1$ set $K_L=B_L(0)\setminus (\cup_{y_0\in S_0}B_{1/L}(y_0))$.
Another covering argument together with \eqref{5.40} allows to bound 
$e^{v_k^2}\le Ce^{v_k^2(y_k^{(0)})}=Ce^{u_k^2(x_k)}$ on $K_L$, where $C=C(L)$.
By \eqref{4.6} and Lemma \ref{lemma4.1} for any $L>0$ we then obtain
\begin{equation*}
\begin{split}
 \int_{K_L}|II_k|^2 dx 
 & \le C \lambda_kr_k^2v_k^2(y_k^{(0)}) e^{v_k^2(y_k^{(0)})}
   \cdot \big(\lambda_k^{-1}\int_{B_L(0)}r_k^2\dot{v}_k^2e^{v_k^2}dx\big)\\
 & = C \lambda_k r_k^2u_k^2(x_k) e^{u_k^2(x_k)}
 \cdot \big(\lambda_k^{-1}\int_{B_{Lr_k}(x_k^{(i)})}\dot{u}_k^2e^{u_k^2}dx\big) 
 \rightarrow 0
\end{split}
\end{equation*}
as $k \rightarrow \infty$. Next rewrite $I_k$ as 
\begin{equation*}
   I_k=\lambda_k r_k^2 v^2_k(y_k^{(0)})e^{v^2_k(y_k^{(0)})}\hat{v}_ke^{\eta_k(\hat{v}_k+1)},
\end{equation*}
where $\hat{v}_k=\frac{v_k}{v_k(y_k^{(0)})}$. From \eqref{conv} we get that 
$\hat{v}_k \rightarrow 1$ locally uniformly on $\mathbb{R}^2\backslash S_0$ while 
from \eqref{sub} we conclude that 
\begin{equation*}
   \lambda_k r_k^2 v^2_k(y_k^{(0)})e^{v^2_k(y_k^{(0)})}
   = \lambda_k r_k^2u_k^2(x_k) e^{u_k^2(x_k)}\rightarrow \mu_0
\end{equation*}
for some $\mu_0>0$ as $k\rightarrow \infty$. Since by Proposition \ref{prop5.2} 
$\eta_k$ is locally uniformly bounded, from \eqref{eq} and the above considerations 
via standard $L^2$-theory we obtain that $\eta_k$ is uniformly locally bounded in $H^2$ 
away from $S_0$. Hence we conclude that $\eta_k$ converges locally uniformly away 
from $S_0$ and weakly locally in $H^2$ to some limit 
$\eta_0\in H^2_{loc}(\mathbb{R}^2\setminus S_0)$ which is smooth away from $S_0$ 
and which satisfies the equation 
\begin{equation}\label{etamu}
  -\Delta \eta_0=\mu_0 e^{2\eta_0}
\end{equation} 
on $\mathbb{R}^2\setminus S_0$. Recalling that 
$\hat{v}_k I_k = \lambda_k r_k^2 v^2_k e^{v^2_k}$, from \eqref{5.1} we can estimate
\begin{equation*}
  \begin{split}
   \int_{\mathbb{R}^2}e^{2\eta_0}dx
   & \le \lim_{L\rightarrow \infty} \liminf_{k\rightarrow \infty} 
   \int_{K_L}\hat{v}_k^2e^{\eta_k(\hat{v}_k+1)}dx 
   = \lim_{L\rightarrow \infty} \liminf_{k\rightarrow \infty} 
   \int_{K_L}\mu_0^{-1}\hat{v}_k I_k\; dx\\
   & \le \mu_0^{-1} \limsup_{k\rightarrow \infty} 
   \int_{\Omega} \lambda_k u_k^2 e^{u^2_k}dx \le C\Lambda
  \end{split}
\end{equation*}
 as before, and $e^{2\eta_0}\in L^1(\mathbb{R}^2)$. 

Similar to \eqref{5.26} we can moreover estimate for every $L\ge 1$
\begin{equation*}
   \int_{B_L(y_0)} |II_k| dx \rightarrow 0\ \ \text{as}\ \ k\rightarrow \infty,
\end{equation*}
and analogous to \eqref{5.31} we have
\begin{equation*}
   \int_{B_{1/L}(y_0)} I_k \; dx \rightarrow 0
\end{equation*}
for any $y_0 \in S_0$ if we let first $k\rightarrow \infty$ and then $L\rightarrow \infty$. 
Hence for such $y_0$ we conclude that
\begin{align*}
   \limsup_{k\rightarrow \infty} \int_{B_{1/L}(0)} |\Delta \eta_k| dx \rightarrow 0 \
   \text{ as } \ L\rightarrow \infty.
\end{align*}
This shows that $\eta_0$ extends as a distribution solution of \eqref{etamu} on all of $\mathbb{R}^2$. 
The claim then follows from the classification result of Chen-Li \cite{Chen-Li}.

In the case of Proposition \ref{prop5.1g} we argue similarly by scaling with $r_k=r_k^{(l+1)}$. 
Note that in this case $S_0=\{0\}$.

\section{Applications}
In this final section we will use Theorem \ref{thm1.1} to obtain solutions to
\eqref{1.01} in the supercritical high energy regime.

Let $\Omega$ be a bounded domain in ${\mathbb R}^2$. 
Recall the Moser-Trudinger inequality
\begin{equation} \label{6.1}
  \sup_{u \in H^1_0 (\Omega); ||\nabla u||^2_{L^2(\Omega)} \le 1}
  \int_{\Omega} e^{4 \pi u^2} \, dx < \infty;
\end{equation}
see \cite{Moser-1971}, \cite{Trudinger-1967}. The exponent $\alpha = 4 \pi$
is critical for this Orlicz space embedding in the sense that for 
any $\alpha > 4\pi$ there holds
\begin{equation} \label{6.2}
  \sup_{u \in H^1_0 (\Omega); ||\nabla u||^2_{L^2(\Omega)} \le 1}
  \int_{\Omega} e^{\alpha u^2} \, dx = \infty.
\end{equation}
Indeed, suppose that $B_R(0)\subset\Omega$. Following Moser \cite{Moser-1971},
for $0 < \rho < R$ consider the functions 
\begin{equation*}
  m_{\rho, R} (x) = \frac{1}{\sqrt{2 \pi}}
    \begin{cases}
      \sqrt{\log \left(\frac{R}{\rho} \right)}, &0 \le |x| \le \rho,\\
      \log \left(\frac{R}{r} \right) \bigg/ \sqrt{\log \left(\frac{R}{\rho}
      \right)}, &\rho \le |x| = r < R,\\
      0, &R \le |x|.
    \end{cases}
\end{equation*}
Note that $||\nabla m_{\rho, R}||^2_{L^2(\Omega)} = 1$,
and for any $\alpha > 4\pi$ we have
\begin{equation} \label{6.5}
  \int_{\Omega} e^{\alpha m_{\rho, R}^2} \, dx \rightarrow \infty \hbox{ as } \rho 
  \rightarrow 0.
\end{equation}

After scaling, \eqref{6.1} gives 
\begin{equation} \label{6.3}
  c_{\alpha} = c_{\alpha}(\Omega) :=
  \sup_{u \in H^1_0 (\Omega); ||\nabla u||^2_{L^2(\Omega)} \le \alpha} E(u) < \infty
\end{equation}
for any $\alpha \le 4\pi$, while for any $\alpha > 4\pi$ from \eqref{6.2} we have
\begin{equation} \label{6.4}
  \sup_{u \in H^1_0 (\Omega); ||\nabla u||^2_{L^2(\Omega)} \le \alpha }
  E(u) = \infty.
\end{equation}
If we normalize $vol(\Omega)=\pi$, the constant $c_{4\pi}(\Omega)$
is maximal when $\Omega = B_1(0) =:B$, as can be seen by symmetrization. Let 
$c_* = c_{4\pi}(B)$.

\subsection{Solutions with ``large'' Moser-Trudinger energy on non-contractible 
domains}
As stated in the Introduction, we obtain the following result in the spirit of 
Coron \cite{Coron}.

\begin{theorem} \label{thm6.1}
For any $c^* > c_*$ there are numbers $R_1 > R_2 > 0$ with the following
property. Given any domain
$\Omega \subset {\mathbb R}^2$ with $vol(\Omega) = \pi$ containing the annulus 
$B_{R_1} \setminus B_{R_2} (0)$ and such that $0 \notin \overline{\Omega}$,
for any constant $c_0$ with $c_{4\pi}(\Omega)< c_0< c^*$ problem \eqref{1.01} 
admits a positive solution $u$ with $E(u) =c_0$.
\end{theorem}

The proof of Theorem \ref{thm6.1} relies on the following observation.

\begin{lemma}\label{lemma6.1}
Let $(u_k)$ be a sequence in $H_0^1(\Omega)$ such that 
\begin{equation*}
  E(u_k) \ge c > c_{4\pi}(\Omega),\
  \int_{\Omega} |\nabla u_k|^2 \, dx \rightarrow 4 \pi \ \hbox{ as } k\rightarrow \infty.
\end{equation*}
Then there exists a point $x_0 \in \overline{\Omega}$ such that 
$|\nabla u_k|^2 \, dx \overset{w^*}{\rightharpoondown} 4\pi \delta_{x_0}$
weakly in the sense of measures as $k\rightarrow \infty$ suitably.
\end{lemma}
\begin{proof} We may assume that $u_k \overset{w}{\rightharpoondown} u$
weakly in $H^1_0(\Omega)$ and pointwise almost everywhere as $k\rightarrow \infty$.
Negating our claim, there exist $\alpha_1, r_1 >0$ with $\alpha_1 < 4\pi$ such that
\begin{equation*}
  \sup_{k \in {\mathbb N}, \; x_1\in \Omega}
  \int_{B_{r_1}(x_1)\cap \Omega} |\nabla u_k|^2 \, dx \le \alpha_1 .
\end{equation*}
But then by a reasoning as in the proof of Lemma 3.3 in \cite{Adimurthi-Struwe} 
we conclude that the functions $e^{u_k^2}$ are uniformly bounded in $L^q$ for 
some $q > 1$, and by Vitali's convergence theorem we have
\begin{equation*}
  E(u) =\lim_{k \rightarrow \infty} E(u_k) \ge c > c_{4\pi}(\Omega).
\end{equation*}
Since $\int_{\Omega} |\nabla u|^2 \, dx \le 4\pi$, the latter 
contradicts \eqref{6.3}, which proves our claim. 
\end{proof}

The proof of Theorem \ref{thm6.1} now is achieved via a saddle-point construction
similar to Section 3.4 in \cite{Struwe2000}.
We may assume that $0 < R_1 < 1/2$. Given such $R_1$, fix $R = R_1/4$.
For each $R_2 < R_1/8 = R/2$, moreover, we let $\tau = \tau_{R_2} \in
C^{\infty}_0 (B_R (0))$ be a cut-off function $0 \le \tau \le 1$
satisfying $\tau \equiv 1$ on $B_{R_2} (0)$ and such that $\tau
\rightarrow 0$ in $H^1 ({\mathbb R}^2)$ as $R_2 \rightarrow 0$. 

For $x_0 \in {\mathbb R}^2$ let $m_{\rho, R, x_0} (x) = m_{\rho, R}(x - x_0)$. 
With a suitable number $0 < \rho < R$ to be determined, for any $x_0$
with $|x_0| = 3R$, any $0 \le s < 1$ then we define
$$
  v_{s, x_0} (x) = m_{s \rho, R, (1 - s) x_0} (x) (1 - \tau (x)) \in 
  H^1_0(B_{R_1} \setminus B_{R_2} (0)).
$$
Provided that $\Omega$ contains the annulus $B_{R_1} \setminus B_{R_2} (0)$, these
functions then also belong to $H^1_0(\Omega)$.

Given $c^* > c_*$, we fix the numbers $0 < \rho < R$, $0 < R_2 < R/2$ so that 
\begin{equation} \label{6.6}
  \inf_{0 < s \le 1,\, |x_0|=3R}\big(\frac12
  \int_{\Omega}(e^{8\pi v_{s,x_0}^2}-1)\, dx \big) > c^*
\end{equation}
for all such domains $\Omega$. This is possible by \eqref{6.5}. 
Fixing such a domain $\Omega$, finally, for any given $c_{4\pi}(\Omega)< c_0< c^*$ 
we let 
$$
  w_{s, x_0} = \sqrt{\alpha_{s, x_0}} v_{s, x_0},
$$
where for each $s, x_0$ the number $\alpha_{s, x_0}$ is uniquely 
determined such that 
$$
  E(w_{s, x_0}) = \frac12 \int_{\Omega} (e^{\alpha_{s, x_0}v_{s, x_0}^2}-1) \, dx = c_0.
$$
Observe that \eqref{6.3} and \eqref{6.6} imply the bounds 
$4\pi < \alpha_{s, x_0} < 8\pi$ for each $s, x_0$, and
\begin{equation}\label{6.7}
    \alpha_{s, x_0} \rightarrow 4\pi \hbox{ as } s\rightarrow 0
\end{equation}
uniformly in $|x_0| = 3 R$ by \eqref{6.5}. 

Let $u_{s, x_0}(t)$ be the solution to the initial value 
problem \eqref{1.1} - \eqref{1.5} with initial data $u_{s, x_0}(0) = w_{s, x_0}\ge 0$. 

\begin{lemma}\label{lemma6.2}
With a uniform constant $\alpha_0 > 4\pi$ there holds
\begin{equation}\label{6.8}
  \sup_{0 < s \le 1, |x_0| = 3 R }
  \int_{\Omega} |\nabla u_{s, x_0}(t)|^2 \, dx \ge \alpha_0  
\end{equation}
for all $0 \le t < \infty$.
\end{lemma}

\begin{proof}
Otherwise by \eqref{1.6} we have $||\nabla u_{s, x_0}(t)||^2_{L^2} \rightarrow 4 \pi$ as 
$t \rightarrow \infty$,
uniformly in $s$ and $x_0$, and from Lemma \ref{lemma6.1} we conclude that 
$$
    \sup_{0 < s \le 1, |x_0| = 3R }\operatorname{dist} (m (u_{s, x_0}(t)), \Omega)
    \rightarrow 0
$$
as $t \rightarrow \infty$, where
$$
  m (u) = \frac{\int_{\Omega} x |\nabla u|^2 \, dx}{\int_{\Omega}|\nabla u|^2 \, dx}
$$
is the center of mass. Moreover, by \eqref{6.7}, \eqref{2.2}, and 
Lemma \ref{lemma6.1} we have
$$
    \sup_{0 < s \le s_0,\, |x_0| = 3 R }
    \operatorname{dist} (m (u_{s, x_0}(t)), \Omega) \rightarrow 0
$$
as $s_0 \rightarrow 0$, uniformly in $t \ge 0$.
Recall that $0 \notin \overline{\Omega}$. Thus, for some sufficiently small number 
$0<s_0<1$ and sufficiently large $T >0$ with a uniform constant $\delta > 0$ we have 
$$
   \inf_{ |x_0| = 3 R } |m (u_{s, x_0}(t))| \ge \delta > 0,
$$
provided that either $0 < s \le s_0$ or $t \ge T$. 
Identifying $\partial B_{3R} (0)$ with $S^1$ and letting 
$\pi_{S^1}(p)=p/|p|$ for $p \in {\mathbb R}^2\setminus \{0\}$,
then for sufficiently small $0<s_0<1$ and sufficiently large $T>0$ we can define a 
homotopy $H = H (\cdot , r) \colon S^1\times ]0,T+1] \rightarrow S^1$ by letting
\begin{equation*}
  H (x_0 , r) =
  \begin{cases}
   \pi_{S^1}(m (u_{r, x_0}(0))), & 0 < r \le s_0,\\
   \pi_{S^1}(m (u_{s_0, x_0}(r-s_0))), &s_0 \le r \le T+s_0,\\
   \pi_{S^1}(m (u_{r-T, x_0}(T))), & T+s_0\le r \le T+1.
    \end{cases}
\end{equation*}
Then clearly $H(\cdot ,T+1) \equiv const$, whereas $H(x_0 ,r) \rightarrow x_0/|x_0|$
as $r \rightarrow 0$, uniformly in $x_0$, which is impossible.
The contradiction proves the claim.
\end{proof}

\begin{proof}[Proof of Theorem \ref{thm6.1}]
For any $t>0$ by Lemma \ref{lemma6.2} there are $0 < s(t) \le 1$, $x_0(t)$ 
with $|x_0(t)|=3R$ such that 
\begin{equation}\label{6.9}
  \int_{\Omega} |\nabla u_{s(t), x_0(t)}(t)|^2 \, dx \ge \alpha_0 > 4 \pi .
\end{equation}
Let $(s_1,x_0)$ be a point of accumulation of $(s(t),x_0(t))$ as $t\rightarrow \infty$.
Note that by \eqref{1.6} for any fixed time $t_0$ we have 
\begin{equation} \label{6.10}
  \begin{split}
   8\pi & > \alpha_{s_1,x_0} = \int_{\Omega} |\nabla u_{s_1, x_0}(0)|^2 \, dx 
   \ge \int_{\Omega} |\nabla u_{s_1, x_0}(t_0)|^2 \, dx\\ 
   & \ge \liminf_{t \rightarrow \infty}\int_{\Omega}|\nabla u_{s(t),x_0(t)}(t_0)|^2 \, dx
   \ge \liminf_{t \rightarrow \infty}\int_{\Omega} |\nabla u_{s(t), x_0(t)}(t)|^2 \, dx
   \ge \alpha_0 > 4 \pi.
  \end{split}
\end{equation}
Fix $u_0 = u_{s_1, x_0}(0)\ge 0$ and let $u(t)$ be the solution to the initial value 
problem \eqref{1.1} - \eqref{1.5} with initial data $u(0) = u_0$ with 
associated parameter $\lambda(t)$. 
We claim that $u(t)$ is uniformly bounded and hence converges to 
a solution $u_{\infty} > 0$ of \eqref{1.01} with 
$$
  \int_{\Omega} |\nabla u_{\infty}|^2 \, dx  > 4 \pi \ \hbox{ and } 
  E(u_{\infty}) = c_0.
$$
This will finish the proof of the Theorem. 

Indeed, suppose by contradiction that 
$u(t)$ blows up as $t \rightarrow \infty$. For a sequence of numbers 
$t_k \rightarrow \infty$ as constructed in Lemma \ref{lemma4.1} then 
as $k \rightarrow \infty$ we have
$\lambda_k:= \lambda(t_k)\rightarrow\lambda_{\infty}\ge 0$; moreover, we may assume that
$u_k:=u(t_k) \overset{w}{\rightharpoondown} u_{\infty}$ in $H^1_0(\Omega)$ and
pointwise almost everywhere, where 
$u_{\infty}$ solves \eqref{1.01}. Finally, Theorem \ref{thm1.1} and \eqref{6.10}
also give the bound
\begin{equation}\label{6.12}
  \int_{\Omega} |\nabla u_{\infty}|^2 \, dx < 4 \pi .
\end{equation}
It then follows that $\lambda_{\infty} = 0$. Indeed, if we assume $\lambda_{\infty} > 0$,
from \eqref{2.2} and the dominated convergence theorem we infer 
$$
  E(u_{\infty})  = \lim_{k \rightarrow \infty} E(u_k) =c_0 > c_{4\pi}(\Omega),
$$
which is impossible in view of \eqref{6.12} and \eqref{6.3}. 
But with $\lambda_{\infty} = 0$ in view of \eqref{1.01} also $u_{\infty}$ must vanish 
identically, and from Theorem \ref{thm1.1} it follows that
\begin{equation}\label{6.13}
  \lim_{k \rightarrow \infty}\int_{\Omega} |\nabla u_k|^2 \, dx = 4 \pi l,
\end{equation}
for some $l \in {\mathbb N}$, contradicting \eqref{6.10}. The proof is complete.
\end{proof}

\subsection{Saddle points of the Moser-Trudinger energy}
Finally, we establish Theorem \ref{thm1.3}. Recall that by \cite{Flucher}, Corollary 7,
on any bounded domain $\Omega \subset {\mathbb R}^2$ 
the Moser-Trudinger energy $E$ attains its maximum $\beta^*_{4\pi} := c_{4\pi}(\Omega)$
in the set $M_{4\pi}$ defined in \eqref{1.00}. Moreover, we have

\begin{lemma}\label{lemma6.4}
The set $K_{4\pi}$ of maximizers of $E$ in $M_{4\pi}$
is compact.
\end{lemma}

\begin{proof}
Any $u \in K_{4\pi}$ solves \eqref{1.01}. 
Given a sequence $(u_k) \subset K_{4\pi}$, we may assume that 
$u_k \rightharpoondown u_{\infty}$ weakly in $H^1_0(\Omega)$ 
as $k \rightarrow \infty$ while by \eqref{2.3} the associated numbers 
$\lambda_k \rightarrow \lambda_{\infty}\ge 0$. If $\lambda_{\infty}> 0$, from \eqref{2.2} 
and the dominated convergence theorem as above we conclude that
$E(u_k) \rightarrow E(u_{\infty})$, so that $E(u_{\infty})=\beta^*_{4\pi}$
and $u_{\infty}\neq 0$. But by a result of P.-L.~Lions \cite{Lions}, Theorem I.6, this
implies that the functions $e^{u_k^2}$ are uniformly bounded in $L^q$ for 
some $q > 1$, and $u_k \rightarrow u_{\infty}$ strongly in $H^1_0(\Omega)$, as claimed.
On the other hand, if $\lambda_{\infty}= 0$, from \eqref{1.01} we conclude that also
$u_{\infty}$ must vanish and $E(u_{\infty})=0$. Theorem I.6 in \cite{Lions} then implies 
weak convergence 
\begin{equation}\label{6.15}
  |\nabla u_k|^2dx \overset{w^*}{\rightharpoondown} 4\pi \delta_{x_0}
\end{equation}
for some $x_0 \in \overline{\Omega}$ in the sense of measures, 
and by Flucher \cite{Flucher}, Lemma 4 and Theorem 5, we have
$E(u_k)<\beta^*_{4\pi}$ for large $k$, contradicting our choice of $(u_k)$.
\end{proof}

In view of Lemma \ref{lemma6.4} now 
Lemma 5.3 from \cite{Struwe84} remains valid for a general domain and there
exist numbers $\alpha^* > 4\pi$, $\varepsilon > 0$ such that for any
$4\pi< \alpha < \alpha^*$ there holds 
\begin{equation*}
  \beta^*_{\alpha} := \sup_{N_{\alpha,\varepsilon}}E > 
  \sup_{N_{\alpha,2\varepsilon}\setminus N_{\alpha,\varepsilon}}E
\end{equation*}
where 
\begin{equation*}
   N_{\alpha,\varepsilon} = \{u \in M_{\alpha}; \
    \exists v \in K_{4\pi}: ||\nabla (u-v)||_{L^2} < \varepsilon\}\, .
\end{equation*}
Moreover, for any such $\alpha$ there exists $\overline{u} \in N_{\alpha,\varepsilon}$ where 
$\beta^*_{\alpha}=E(\overline{u})$ is attained, and $\overline{u}$ solves \eqref{1.01}
for some $\overline{\lambda} \ge 0$. By \eqref{6.4} the set 
\begin{equation*}
  \Gamma_{\alpha} = \{\gamma \in C^0([0,1[;M_{\alpha}); 
  \gamma(0) = \overline{u},\, E(\gamma(1)) > \beta^*_{\alpha} \},
\end{equation*}
then is non-void for any $4\pi < \alpha < \alpha^*$. Since any $\gamma \in\Gamma_{\alpha}$
necessarily passes through the set $N_{\alpha,2\varepsilon}\setminus N_{\alpha,\varepsilon}$ we have
\begin{equation}\label{6.16}
  \beta_{\alpha} := \sup_{\gamma \in \Gamma_{\alpha}} \inf_{0 < s < 1} E(\gamma(s)) 
  < \beta^*_{\alpha}.
\end{equation}
Finally, observing that 
\begin{equation*}
  c_{\alpha}=\sup_{u \in M_{\alpha}}E(u) \rightarrow 0 \hbox{ as }
  \alpha \rightarrow 0 \, ,
\end{equation*}
we can choose $4\pi < \alpha_1 < \alpha^*$ such that 
\begin{equation}\label{6.17}
  c_{\alpha-4\pi} < \beta_{\alpha} \hbox{ for all } \alpha \in ]4\pi,\alpha_1[.
\end{equation}
Clearly, we may assume that $\alpha_1 \le 8\pi$. 

\begin{proof}[Proof of Theorem \ref{thm1.3}]
Let $4\pi< \alpha < \alpha_1$. It remains to find $\underline{u}$. Fix some 
$\gamma \in \Gamma_{\alpha}$ with 
\begin{equation*}
    \inf_{0 < s < 1} E(\gamma(s))> c_{\alpha-4\pi}.
\end{equation*}

Fix a number $\beta$ with $\beta_{\alpha} < \beta < \beta^*_{\alpha}$. 
As long as $E(u(s,t)) \le \beta$ let $u(s,t)\ge 0$ be the solution to the initial 
value problem \eqref{1.1}, \eqref{1.2}, \eqref{0.1} with initial data 
$u(s,0) = \gamma(s)\ge 0$, and let $u(s,t)=u(s,t(s))$ for 
all $t \ge t(s)$ if there is some first $t(s) \ge 0$ where $E(u(s,t(s))) = \beta$. 
Note that by the implicit function theorem the family $u(s,t)$ thus defined depends
continuously both on $s$ and $t$ unless $u_t(s,t(s))=0$ for some $s$ with 
$E(u(s,t(s))) = \beta$, that is, unless there is a solution $0 < u \in M_{\alpha}$ 
of \eqref{1.01} with $E(u) = \beta$, in which case the proof is complete.

For $t>0$ let $0\le s(t) <1$ be such that 
\begin{equation*}
  E(u(s(t),t)) = \inf_{0 < s < 1}E(u(s,t))\le \beta_{\alpha}
\end{equation*}
and let $s_1$ be a point of accumulation of $(s(t))_{t>0}$ as $t\rightarrow \infty$.
Note that similar to \eqref{6.10} by \eqref{0.2} for any fixed time $t_0$ we have 
\begin{equation*}
  \begin{split}
   E(u(s_1,t_0)) \le \limsup_{t \rightarrow \infty}E(u(s(t),t_0)) 
   \le \limsup_{t \rightarrow \infty}E(u(s(t),t)) \le \beta_{\alpha}.
  \end{split}
\end{equation*}
Fix $u_0 = \gamma(s_1)\ge 0$ and let $u(t)$ with associated parameter $\lambda(t)$
be the solution to the initial value problem \eqref{1.1}, \eqref{1.2}, \eqref{0.1}
with initial data $u(0) = u_0$, satisfying 
\begin{equation} \label{6.18}
   c_{\alpha-4\pi} < E(\gamma(s_1)) = E(u(0)) \le E(u(t)) 
   \le \beta_{\alpha} < \beta < \beta^*_{\alpha} \quad
   \hbox{for all } t. 
\end{equation}
We claim that $u(t)$ is uniformly bounded and thus converges to 
a solution $0<u_{\infty}\in M_{\alpha}$ of \eqref{1.01} with
$0 < E(u_{\infty}) < \beta^*_{\alpha}$.
For this we argue as in the  proof of Theorem \ref{thm6.1}.

Indeed, suppose by contradiction that 
$u(t)$ blows up as $t \rightarrow \infty$. For a sequence of numbers 
$t_k \rightarrow \infty$ as constructed in Lemma \ref{lemma4.1} then 
as $k \rightarrow \infty$ we have
$\lambda_k:= \lambda(t_k)\rightarrow\lambda_{\infty}\ge 0$; moreover, we may assume that
$u_k:=u(t_k) \overset{w}{\rightharpoondown} u_{\infty}$ in $H^1_0(\Omega)$ and
pointwise almost everywhere, where 
$u_{\infty}\ge 0$ solves \eqref{1.01} with $||\nabla u_{\infty}||_{L^2}^2 \le \alpha-4\pi$
in view of Theorem \ref{thm1.1}. 
But then $\lambda_{\infty} = 0$.
Indeed, if $\lambda_{\infty} > 0$, from \eqref{0.5} and the dominated convergence 
theorem we infer $E(u_k) \rightarrow E(u_{\infty})\le c_{\alpha-4\pi}$, 
contradicting \eqref{6.18}.
But with $\lambda_{\infty} = 0$ in view of \eqref{1.01} also $u_{\infty}$ must vanish 
identically, and Theorem \ref{thm1.1} yields the contradiction $\alpha = 4 \pi$.
The proof is complete.
\end{proof}

\end{document}